\documentclass[leqno,12pt,a4paper]{article}
\usepackage[a4paper]{geometry}% vazno - za dimenziju papira u pdflatexu, inace graphicx
\usepackage{xcolor}
\usepackage{graphicx, xcolor}
\usepackage{tikz}
\usepackage{epstopdf}
\usepackage[utf8]{inputenc}
\usepackage[toc,page]{appendix}
\usetikzlibrary{arrows}
\usepackage[linesnumbered,ruled]{algorithm2e}
\usepackage{multicol,tcolorbox}
\usepackage{subcaption}
\usepackage{algpseudocode}
\usepackage{bm}

\usepackage{hyperref}
\hypersetup{
    colorlinks=true,
    linkcolor=blue,
    filecolor=magenta,      
    urlcolor=cyan,
}

\usepackage{amsfonts,amssymb,amsthm,amsmath,amsbsy,bm}

\usepackage[english]{babel}
\usepackage{csquotes}
\usepackage[backend=biber, giveninits=true, url=false, doi=false, eprint=false, isbn=false]{biblatex}
\DeclareFieldFormat[article]{title}{#1}
\DeclareFieldFormat[inproceedings]{title}{#1}
\DeclareFieldFormat[incollection]{title}{#1}
\DeclareFieldFormat[article]{pages}{#1}
\usepackage[textsize=tiny,color=green!40]{todonotes}	% Comments in the box

\def\mA{{\bf{A}}}
\def\mM{{\bf{M}}}
\def\mI{{\bf{I}}}
\def\Lb#1{{{\rm L}^\infty(#1)}}
\def\vu{{\sf u}}
\def\vw{{\sf w}}

\def\vH{{\sf H}}
\def\vp{{\sf p}}

\def\vv{{\sf v}}
\def\vz{{\sf z}}
\def\Hmj#1{{{\rm H}^{-1}(#1)}}

\def\lK{{\mathcal K}}
\def\lA{{\mathcal A}}
\def\lL{{\mathcal L}}

\def\pL#1#2{{{\rm L}^{#1}(#2)}}

\def\Hj#1{{{\rm H}^{1}(#1)}}
\def\Hmj#1{{{\rm H}^{-1}(#1)}}
\def\Kth{\lK(\theta)}
\def\Kthn{\lK(\theta_n)}
\def\eps{\varepsilon}

\newcommand{\lO}{\mathcal O}

\DeclareMathOperator{\MAP}{\Phi_\psi}
\DeclareMathOperator{\sh}{{\rm sh}}
\DeclareMathOperator{\ch}{{\rm ch}}

\definecolor{minered}{RGB}{139, 1, 1}
\definecolor{minegrey}{RGB}{150, 150, 150}

\DeclareMathOperator{\I}{\hspace{-0.05em}\text{\bf I}\hspace{-0.05em}}

\DeclareMathOperator{\ddiv}{div}
\DeclareMathOperator{\HH}{H}

\DeclareMathOperator{\ID}{Id}

\DeclareMathOperator{\WW}{W}

\let\dv\ddiv
\newcommand{\R}{\mathbb R}
\newcommand{\N}{\mathbb N}

\newcommand{\mx}{{\bm x}}

\newcommand{\mn}{{\bm n}}
\newcommand{\mr}{{r}}

\newcommand{\ssubset}{\subset\joinrel\subset}
\newcommand{\JJ}{\mathcal J}

\newcommand{\mxi}{{\bm \xi}}
\newcommand{\my}{{\bm y}}
\newcommand{\mz}{{\bm z}}
\newcommand{\Kthtil}{\lK(\tilde\theta)}

\newcommand{\Lj}[1]{{{\rm L}^1(#1)}}
\newcommand{\Ld}[1]{{{\rm L}^2(#1)}}
\newcommand{\Ldl}[1]{{{\rm L}_{\rm loc}^2(#1)}}
\newcommand{\Hjn}[1]{{{\rm H}_0^{1}(#1)}}
\newcommand{\Cbc}[1]{{\rm C}^\infty_c(#1)}
\newcommand{\supp}{\operatorname{supp}}
\newcommand{\vnul}{{\sf 0}}

\newtheorem{theorem}{Theorem}[section]

\newtheorem{example}[theorem]{Example}

\newtheorem{lemma}[theorem]{Lemma}
\newtheorem{prop}[theorem]{Proposition}

\newtheorem{remark}[theorem]{Remark}
\newtheorem{algorithmm}[theorem]{Algorithm}
\definecolor{redd}{RGB}{255,0,0}

\begin{document}
% title
\title{Hybrid Optimization Techniques for Multi-State Optimal Design Problems}
%Hybrid numerical method based on homogenization and shape derivative
\date{}

\author{Marko Erceg\footnote{Department of Mathematics, Faculty of Science, University of Zagreb, Croatia, \texttt{maerceg@math.hr}}, Petar Kun\v stek\footnote{Department of Mathematics, Faculty of Science, University of Zagreb, Croatia, \texttt{petar@math.hr}}, Marko Vrdoljak \footnote{Department of Mathematics, Faculty of Science,  University of Zagreb, Croatia, 
\texttt{marko@math.hr}} }

\maketitle

\begin{abstract}

This paper addresses optimal design problems governed by multi-state stationary diffusion equations, 
aiming at the simultaneous optimization of the domain shape and the distribution of two isotropic 
materials in prescribed proportions. Existence of generalized solutions is established via a hybrid 
approach combining homogenization-based relaxation in the interior with suitable restrictions on admissible domains. 

Based on this framework, we propose a numerical method that integrates homogenization and shape optimization. 
The domain boundary is evolved using a level set method driven by the shape derivative, while the interior 
material distribution is updated via an optimality criteria algorithm. The approach is demonstrated on a representative example.
\end{abstract}



\section{Introduction}

Shape optimization refers to the determination of an optimal domain that minimizes a prescribed objective functional, usually defined by the solution of a partial differential equation (PDE) known as the state equation. As such, shape optimization can be considered as a branch of distributed control theory, where the domain acts as a control variable.

A fundamental difficulty in this field lies in the nonexistence of classical optimal shapes. Minimizing sequences of domains tend to develop oscillations, holes, or degeneracies, so that no Lipschitz domain attains the infimum of the cost functional. This nonexistence has profound implications: numerical algorithms may fail to converge or depend strongly on the initial guess. Two principal strategies have emerged to address this: (i) restricting the class of admissible domains to ensure compactness and existence, or (ii) relaxing the problem by enlarging the admissible set  where existence can be recovered.

The homogenization theory plays a central role in obtaining this relaxation \cite{MT85, KS}, by introducing a notion of generalized  designs that ensures the optimization problem remains well-posed and physically meaningful. This, in turn, leads to more stable and efficient numerical computations of relaxed designs \cite{BendKik,All02,BS03}.

On the other hand, imposing additional (uniform) regularity assumptions on the class of admissible domains can ensure the existence of solutions for certain shape optimization problems. Examples include the uniform cone condition \cite{Chenais, P84}, a bounded number of connected components of the complement in two-dimensional cases \cite{Sverak}, perimeter constraints \cite{AB93}, uniform continuity conditions \cite{LNT02}, the uniform cusp property \cite{DZ}, and uniform boundedness of the density perimeter \cite{BZ96}. Further theoretical insights are provided in \cite{SZ92,HP18,BB05,NTS06}.

This theoretical development was accompanied by a numerical treatment of the problem using shape calculus, building on the pioneering work of Hadamard \cite{MS76, P84,SZ92,AConc}. 
Numerical methods in shape optimization differ mainly in how they represent geometry. Some approaches, such as the level set and phase-field methods, describe shapes implicitly using 
a scalar field that marks the region occupied by the material.  Other approaches represent shapes explicitly through a computational mesh or CAD model, allowing for accurate analysis but 
requiring frequent remeshing or geometric updates as the shape evolves.

%Numerical methods in the literature are mostly concerned with determining an unknown domain occupied by a given material. 
%Such problems can, beside shape derivative techniques, be effectively addressed using the homogenization method, which accounts for fine mixtures of the original material and voids. 

In this paper, we consider the problem of determining not only the optimal distribution of two phases but also the optimal shape of the region they occupy. 
The problem has been addressed in the literature using alternative approaches. Extensions of the SIMP scheme to multiple phases were proposed in 
\cite{ST97,GS00}, considering three-phase mixtures (two materials plus void). In \cite{WW04, ADDM14}, the level set method was applied to multi-phase 
shape optimization, also allowing one of the phases to represent void. 

Motivated by the robustness of the homogenization method, our approach combines homogenization-based relaxation  within the domain with shape optimization based on the shape derivative of the overall region. 
 For this reason, we refer to it as a hybrid method.
% The proposed method offers potential advantages due to the robustness of the 
%homogenization approach \cite{All02}, which provides a proper relaxation of the original problem. 
Moreover, the optimality criteria method is known 
to produce accurate approximations within only a few iterations, which  facilitates the application of the shape derivative method 
to the external boundary. 

The resulting optimal design can, through penalization techniques, be reduced to a classical design with pure phases, or it can serve as an informed initial guess for the subsequent application of the shape derivative method to the distribution of the two phases within the “optimal” domain.
In the numerical example presented in the final section, the penalization step was not implemented, since the optimal design produced by the proposed method turned out to be classical.
This is due to the radial symmetry of the problem: the resulting design is known to be optimal for the fixed-domain problem posed on the resulting domain \cite{KunVrd20}.

We now proceed to formulate the specific 
two-material optimal design problem under consideration.

Let $D\subset\R^d$, $d\geq 2$, be a bounded open fixed set and  $\Omega \subset D$ a Lipschitz domain. The open set $\Omega$ is made of two materials with isotropic conductivities $0<\alpha<\beta$, so its overall conductivity is given by
\begin{equation}\label{A_clas}
\mA = \chi\alpha \mI +(1-\chi)\beta \mI,
\end{equation}
where $\chi\in\Lb{\Omega;\{0,1\}}$  represents the characteristic function of the less-conductive phase. For different right-hand sides $f_1,\ldots f_m\in \Ld D$ we denote  the corresponding temperatures $\vu:=(u_1,\ldots, u_m)\in\Hj{\Omega;\R^m}$ of the body, as the unique solutions of the following boundary value problems
\begin{equation}\label{state1}%\tag{S}
 \left\{ \begin{array}{cc} -\dv(\mA \nabla u_i)=f_i& \text{ in }\Omega\\ 
u_i=h_i &\text{ on }\Gamma_0\\
u_i=0 &\text{ on }\Gamma_1\\
 \mA\nabla u_i\cdot \mn=0 &\text{ on }\Gamma_2. \end{array} \right.
\end{equation}

Here, the boundary $\partial\Omega$ consists of three  parts: $\partial \Omega=\Gamma_0\dot\cup\Gamma_1\dot\cup\Gamma_2$, where 
$\Gamma_0$ is relatively open and fixed, and $h_i\in \HH^{\frac12}(\Gamma_0)$
are prescribed.
The sets $\Gamma_1$ and $\Gamma_2$ represent subsets of 
$D$ on which the Dirichlet and Neumann boundary conditions, 
respectively, can be applied.
Note that simply requiring $h_i \in \HH^{\frac12}(\Gamma_0)$ is insufficient 
to ensure well-posedness; their choice must be consistent with the 
other boundary conditions, particularly the homogeneous Dirichlet 
condition imposed on the segment $\Gamma_1$.
We address this specification separately within each of the three settings in Section \ref{sec:relax}.
%For instance, if $\Gamma_2$ is empty, we would take $h_i$ from the space $\HH^{\frac12}_{00}(\Gamma_0)$ (see \cite{LM72}), and if $\Gamma_1$ is empty, the space $\HH^{\frac12}(\Gamma_0)$
%is appropriate. In the general case, we impose more specific assumptions on the possible choices of $\Gamma_1$ and $\Gamma_2$, and, consequently, on the functions $h_i$.

\begin{remark}
The functions $f_i$, for $i=1,\dots,m$, could potentially belong to a broader space, specifically the dual of the space $\bigl\{v\in\Hj D : v|_{\Gamma_0\cup\Gamma_1}=0\bigr\}$, where the restriction is naturally understood in the sense of traces. However, for the sake of simplicity in the subsequent analysis, and anticipating the even stricter requirements of the numerical application to follow, we shall work with a subspace consisting of square integrable functions on $D$ (i.e.~$f_i \in \Ld D$).
\end{remark}

 The aim is to find a domain $\Omega$, such that $\Gamma_0\subseteq\partial\Omega$, and a  characteristic function $\chi$ that minimizes the objective functional 
\begin{align*}
{\mathcal I}(\Omega,\chi)=\int\limits_\Omega  \chi(\mx ) g_{\alpha}(\mx ,\vu(\mx ))+ (1-\chi(\mx )) g_{\beta}(\mx ,\vu(\mx )) \,\mathrm{d}\mx .
\end{align*} 

The amount of the first phase is also given:
\begin{align*}
    \int\limits_{\Omega}\chi \,\mathrm{d}\mx = q_\alpha,
\end{align*}
as well as the volume (i.e.~the Lebesgue measure) of the domain: $|\Omega|=V$.
There is no definitive reason to expect minimizers; however, if they do exist, we refer to this as a classical design.

We provide a brief overview of the paper. In Section \ref{sec:relax}, we introduce a hybrid approach for establishing the existence of a 
(generalized) solution. Relaxation via the homogenization method is applied within the domain, while the class of admissible domains is suitably restricted. 
We first examine the cases where either $\Gamma_1$ or $\Gamma_2$ is empty, and then extend the analysis to the general case by combining the techniques 
developed for these two settings. Section \ref{sec:algorithm} presents the numerical approach corresponding to this hybrid framework. The optimality criteria 
method is applied inside the domain, while the domain boundary is updated using the shape derivative of the cost functional. Section \ref{sec:num_ex} 
illustrates the approach with a numerical example.

\section{Relaxation}\label{sec:relax}

\subsection{Preparation}

If $\Omega$ is fixed, achieving proper relaxation involves introducing  generalized materials, which are mixtures of two original phases on a micro-scale, introduced by Murat and Tartar \cite{MT85}. In this relaxation process,  characteristic functions $\chi\in\Lb{\Omega;\{0,1\}}$ are replaced by 
local fractions $\theta\in\Lb{\Omega;[0,1]}$, and \eqref{A_clas} by $\mA\in\Kth$ almost everywhere on $\Omega$, where $\Kth$ denotes the set of all possible effective conductivities  obtained by mixing original phases with local fraction $\theta$. Under  suitable growth conditions on $g_\alpha$ and $g_\beta$ (see (A5) below) the relaxed problem reads
\begin{equation}\label{HomogRelax}
   \left\{ \begin{array}{cc}
         {\JJ}(\Omega,\theta,\mA)=\int_\Omega  \theta(\mx ) g_{\alpha}(\mx ,\vu(\mx ))+ (1-\theta(\mx )) g_{\beta}(\mx ,\vu(\mx )) \,\mathrm{d}\mx \to \min \\[0.2cm]
          \theta\in\Lb{\Omega;[0,1]}, \, \int_\Omega\theta\,d\mx=q_\alpha, \, \mA\in \Kth \text{ a.e.~in } \Omega, \,\vu \text{ solves \eqref{state1} }.
    \end{array}\right.
\end{equation}

The determination of the set $\Kth$ is reffered to as the G-closure problem (see \cite[Subsection 2.1.2]{All02}). For mixtures of two isotropic phases, this problem admits 
a complete characterization: $\Kth$ consists of all symmetric matrices whose eigenvalues satisfy the following inequalities 
(see \cite{MT85,LC86}) 
\begin{equation}\label{eq:Gclosure}\displaystyle
   \left\{\quad \begin{array}{cc}
         \displaystyle\lambda_\theta^- \leq \lambda_j \leq \lambda_\theta^+\quad j=1,\hdots,d   \\
         \displaystyle \sum_{j=1}^d \frac{1}{\lambda_j -\alpha} \leq \frac{1}{\lambda_\theta^--\alpha}+ \frac{d-1}{\lambda_\theta^+-\alpha}\\
         \displaystyle \sum_{j=1}^d \frac{1}{\beta-\lambda_j} \leq \frac{1}{\beta-\lambda_\theta^-}+ \frac{d-1}{\beta-\lambda_\theta^+}
    \end{array}\right.
\end{equation}
where $\lambda_\theta^-=(\theta/\alpha+(1-\theta)/\beta)^{-1}, \: \lambda_\theta^+=\theta\alpha+(1-\theta)\beta.$ 
This precise characterization of $\Kth$ is relevant only for the numerical approach and is not required for the  results presented in this section. 
These results equally apply to the case of mixing more than two materials, including anisotropic ones \cite{Ta1, Ta2}.

To obtain existence of optimal $\Omega$, we shrink the set of admissible domains, by assuming the well-known cone condition.

For  $\my,\mxi\in \R^d$, $\mxi$ a unit vector, and $\varepsilon >0$ we denote by $C(\my,\mxi,\varepsilon)$ the cone of vertex $\my$, of direction $\mxi$ and dimension $\varepsilon$, i.e. 
  \begin{equation}\label{eq:eps-cone}
  C(\my,\mxi,\varepsilon)=\left\{ \mz\in\R^d :\begin{array}{cc}
      (\mz-\my) \cdot \mxi \geq |\mz-\my|\cos\varepsilon ,  \\
        0<|\mz-\my|<\varepsilon
  \end{array}\right\} \,.
  \end{equation}
  An open set $\Omega$ is said to have \emph{$\varepsilon$-cone property} if (v. Figure \ref{fig:eps-cone})
  \[
  (\forall \mx\in\partial\Omega)(\exists\mxi_\mx\in\R^d, |\mxi_\mx|=1)(\forall \my\in\overline\Omega\cap B(\mx,\varepsilon)) \quad C(\my,\mxi_\mx,\varepsilon)\subset\Omega \,.
  \]
\begin{figure}
    \centering
    \includegraphics[width=0.4\textwidth]{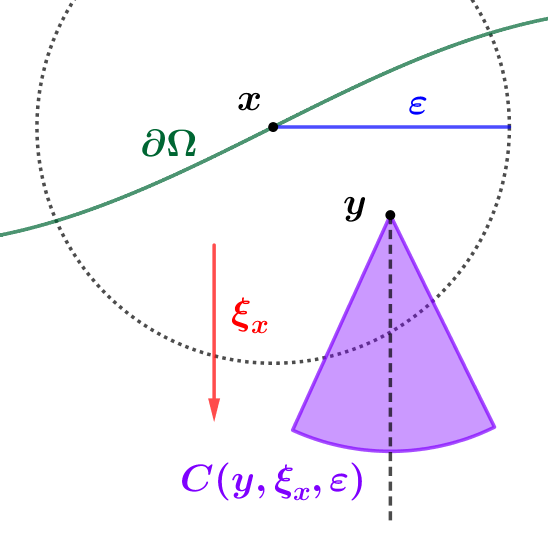}
    \caption{The $\eps$-cone property of a set}
    \label{fig:eps-cone}
\end{figure}%
Here $B(\mx,\varepsilon)$ denotes the open ball centred at $\mx$
with radius $\varepsilon$, while $\overline{\Omega}$ stands for the closure of the set $\Omega$.
This property, when considered for a single domain $\Omega$, is commonly referred to as the uniform cone property \cite{Grisvard}. However, in problems involving a family of domains $\Omega$, emphasizing the uniformity of the parameter $\varepsilon$ across the entire family is crucial, which makes the term $\varepsilon$-cone property well-established \cite{HP18}. It is a well-known result that the $\varepsilon$-cone property is equivalent to the (uniform) Lipschitz condition on the boundary of the domains \cite{Chenais} (see also \cite[Theorem 2.4.7]{HP18}).

Let us now introduce the set of relaxed admissible designs. We begin by addressing the standing assumptions that are required to define these designs. For the reader's convenience, we will also use the context of these assumptions to fix and define some notation that will be used throughout the paper.
\begin{itemize}
\item[(A1)] $D\subset \R^d$ is an open and bounded set (reference domain), where $d\geq 2$ is a fixed integer.
\item[(A2)] $0<\alpha<\beta$ (material conductivities), $V>0$ (volume constraint), $q_\alpha\in (0,V)$ (quantity of the first phase), $\eps>0$ (regularity of domains) are fixed.
\item[(A3)] Let $M \subset \overline{D}$ be a fixed, compact $(d-1)$-dimensional Lipschitz manifold with boundary. We define $\Gamma_0$ 
as the manifold interior of $M$, namely $\Gamma_0 = M \setminus \partial M$, where $\partial M$ denotes the manifold boundary of $M$. 

Then, the set $\Gamma_0$ admits a finite decomposition into connected components $\Gamma_0^i$, $i=1,2,\dots, N$, whose closures are 
mutually disjoint: for any $i,j\in\{1,2,\dots,N\}$, $i\neq j$, we have $\overline{\Gamma_0^i}\cap\overline{\Gamma_0^j}=\emptyset$.
{ In particular, each $\Gamma_0^i$ is an open $(d-1)$-dimensional Lipschitz manifold whose boundary 
$\partial\Gamma_0^i=\overline{\Gamma_0^i}\setminus \Gamma_0^i$ consists precisely of those points at which the component meets the manifold boundary
$\partial M$.}
%Moreover, we require that $\Gamma_0$ has finitely many connected components such that their closures are mutually disjoint.
%$\Gamma_0\subset\R^d$ has finitely many connected components such that their closures are mutually disjoint, i.e.~if $\Gamma_0^i$, $i=1,2,\dots,N$, are connected components of $\Gamma_0$, then for any $i,j\in\{1,2,\dots,N\}$, $i\neq j$, we have $\overline{\Gamma_0^i}\cap\overline{\Gamma_0^j}=\emptyset$.
\item[(A4)] There exists an open and connected $\widehat\Omega\subseteq D$ satisfying the $\eps$-cone property, $|\widehat\Omega|=V$, and $\Gamma_0\subseteq\partial\widehat\Omega$,
where $|\widehat\Omega|$ represents the Lebesgue measure of the set $\widehat\Omega$.  
\end{itemize}

The admissible set, i.e.~the set of \emph{admissible designs}, 
is then defined as
\begin{equation}\label{assump1}
\begin{aligned}
      {\mathcal{A}} =\Biggl\{&(\Omega,\theta,\mA): \Omega\subseteq D \text{ is open, connected and satisfies the $\varepsilon$-cone property}, \\
        & |\Omega|=V, \; \Gamma_0\subset \partial\Omega, \; \theta\in\Lb{\Omega;[0,1]}, \; \int_\Omega \theta\, d\mx =q_\alpha, \; \mA\in \Kth \: \text{ a.e. in }\Omega\Biggr\} \,.
\end{aligned}
\end{equation}
The set $\mathcal{A}$ depends on all parameters introduced in (A1)--(A3), but for simplicity we shall not include that explicit dependence in the notation. 
Note that by (A4) the set of admissible designs $\mathcal{A}$ is non-empty. Indeed, $(\widehat{\Omega},\chi,\mA)$, with characteristic function  $\chi$ on $\widehat{\Omega}$ satisfying 
constraint on the amount of the first phase, and $\mA$ given by \eqref{A_clas} belongs to $\mathcal{A}$.
Let $\mathcal{O}$ denote the projection onto the first component 
of $\mathcal{A}$, i.e.~$\Omega \in \mathcal{O}$ if and only if 
there exist $\theta$ and $\mathbf{A}$ such that 
$(\Omega, \theta, \mathbf{A}) \in \mathcal{A}$. 
The elements of $\mathcal{O}$ will be referred to as \emph{admissible domains}.

The relaxed optimal design problem can be stated as follows:
\begin{equation}\label{RelaxedProblem}
   \left\{ \begin{array}{ll}
         {\JJ}(\Omega,\theta,\mA)\to \min \\[0.2cm]
          (\Omega,\theta,\mA)\in{\mathcal{A}},\text{ and } \vu \text{ solves \eqref{state1} with }\mA.
    \end{array}\right.
\end{equation}

The functions $g_\alpha$ and $g_\beta$ are subject to the following assumptions.
\begin{itemize}
	\item[(A5)]  $g_{\alpha},g_{\beta}$ are Carathéodory functions on $D\times\R^m$ (measurable in $\mx$ for every $\vu$, and continuous in $\vu$ for almost every $\mx$) satisfying the growth condition
	\begin{equation*} %\label{assump2}
	|g_{\gamma}(\mx,\vu)|\leq \varphi_{\gamma}(\mx)+\psi_{\gamma}(\mx)|\vu |^q, \quad \gamma=\alpha,\beta, 
	\end{equation*}
	for some $q\in[1,q^\ast\rangle$ and $\varphi_{\gamma}\in\pL {s_1}D$, $\psi_{\gamma}\in\pL {s_2}D$, 
	where $s_1>1$ and
	\begin{equation*} %\label{eq:qast}
	q^\ast=\left\{\begin{array}{rl} 
	+\infty, & d= 2\\ 
	\frac{2d}{d-2}, & d>2\,,
	\end{array}\right.
	\qquad
	s_2>\left\{\begin{array}{rl} 
	1, & d= 2\\ 
	\frac{q^\ast}{q^\ast-q}, & d>2\,.
	\end{array}\right. \,
	\end{equation*}
\end{itemize}

In the rest of this section, we prove the existence of a solution to the relaxed problem ($\ref{RelaxedProblem}$), under a suitable assumption on $h_i$. The analysis is structured into three steps related to the boundary conditions of the state problem, as defined by the boundary decomposition in $\eqref{state1}$. We first study the pure Dirichlet problem, where $\Gamma_2 = \emptyset$, then we examine the mixed problem with a fixed Dirichlet part, where $\Gamma_1 = \emptyset$, and as a final step, we address the general setting, which covers both previously addressed cases. We present these three distinct settings sequentially to enhance accessibility and guide the reader step-by-step through the technical difficulties. This approach keeps the proofs reasonably short, as in the second and third cases, we emphasize only the new technical components without repeating the parts that carry over from the preceding analysis.

A final technical result in this section concerns the geometry of $\Gamma_0$.
\begin{lemma}\label{lem:O}
	Let us assume {\rm (A1)}--{\rm (A4)}. There exist $N$ open and bounded disjoint subsets $O^i\subset \R^d$, $i=1,2,\dots, N$, such that for every $i\in\{1,2,\dots,N\}$ the following hold:
	\begin{itemize}
		\item[{\rm i)}] $\Gamma_0^i\subset O^i$ and $\partial\Gamma_0^i\subset \partial O^i$;
		\item[{\rm ii)}] $O^i = O^i_- \,\dot\cup\, \Gamma_0^i \,\dot{\cup}\, O^i_+$, where $O^i_\pm$ are { connected} open sets;
		\item[{\rm iii)}] The boundary of any admissible domain $\Omega\in\mathcal{O}$ satisfies $\partial\Omega\cap O^i = \Gamma_0^i$, where $\mathcal{O}$ is the set of admissible domains defined in the above.
	\end{itemize}
\end{lemma}

\begin{proof}
%Let us present only a sketch of the construction of sets $O^i$. 
By \cite{Chenais} (see also \cite[Theorem 2.4.7]{HP18}) any admissible 
domain $\Omega\in\mathcal{O}$ has a Lipschitz boundary. More precisely, in terms of \cite[Definition 2.4.5]{HP18}, the parameters in the definition of a Lipschitz boundary are $L,a,r\in (0,\infty)$, and they depend only on $\eps$.
We define
\begin{equation*}
\delta=\delta(\eps) := \frac{\min\{r,a,1\}}{\sqrt{1+L^2}} \,.
\end{equation*}

Let us take an arbitrary $i\in\{1,2\dots,N\}$. Then we define
\begin{equation*}
O^i := \bigcup_{x\in \Gamma_0^i} B(x,\eta_x) \,,
\end{equation*}
where $\eta_x=\delta\min\{1,d(\partial\Gamma_0^i,x)\}$. %, and $B(x,\eta_x)$ denotes the open ball around $x$ with radius $\eta_x$. 
Then clearly we have that $O^i$ is open and that i) is satisfied. 

Using the definition of a Lipschitz boundary \cite[Definition 2.4.5]{HP18} it is easy to show that for any admissible domain $\Omega\in\mathcal{O}$ we have 
$\partial\Omega \cap B(x,\eta_x) = \Gamma_0^i \cap B(x,\eta_x)$ (here the choice of $\delta$ is fundamental).
This immediately implies that iii) is fulfilled. 

This completes the proof. 
\end{proof}

An immediate consequence of the previous lemma is that $\Gamma_0$ is relatively
open in $\partial\Omega$ for any admissible domain $\Omega\in\mathcal{O}$.

\begin{remark}\label{rem:O}
The interpretation of the open set $O:=\cup_{i=1}^N O^i$ is that the boundary of all admissible domains $\Omega\in\mathcal{O}$ is fixed within $O$. 
Hence, { $\partial\Omega\setminus\Gamma_0$} is contained in $\overline{D}\setminus O$. One could impose such a condition in the definition of $\mathcal{O}$, but the previous lemma precisely asserts that this is implicitly implied by the existing assumptions and is therefore not required as an a priori condition.		
However, we cannot claim that the entire admissible domain $\Omega$ is fixed within $O$. Instead, for any given $i \in \{1, 2, \dots, N\}$, 
the domain $\Omega$ intersects $O^i$ in one of two specific ways: 
\begin{equation*}
\Omega\cap O^i = O^i_- \quad \hbox{or} \quad \Omega\cap O^i = O^i_+ \,.
\end{equation*}

{ We also allow the fixed boundary to intersect the exterior boundary of $D$, 
i.e.~$\Gamma_0 \cap \partial D \ne \emptyset$. Then, some sets $O^i$ are not  strictly contained in $D$ 
($O^i \not\subseteq D$), so that,  of the two equalities above, only one holds, and it is the same for every $\Omega \in \mathcal{O}$.}
%However, this causes no issues for the subsequent analysis. 
\end{remark}

%\subsection{Dirichlet boundary conditions}
\subsection{Dirichlet-type free boundary}

Let us start by imposing the condition on functions $h_i$. 
\begin{itemize}
\item[(H1)] For any $i\in\{1,2,\dots, m\}$ there exists $H_i\in \HH^1_0(O)$ such that $h_i=H_i|_{\Gamma_0}$,
where $O=\cup_{j=1}^N O^j$, and $O^j$ are given by Lemma \ref{lem:O}.
\end{itemize} 

We obtain the following existence result.

\begin{theorem}\label{tm:dir}
%\mbox{}\\
  Under assumptions {\rm (A1)}--{\rm (A5)} and {\rm (H1)}, the optimal design problem \eqref{RelaxedProblem} admits a solution in the case $\Gamma_2=\emptyset$.
\end{theorem}

\begin{proof}
%prebačeno u Step I
%Let us assume that $g=0$. We shall return to the case of general non-homogeneous boundary % conditions at the end of the proof. 

\medskip

\noindent\textbf{Step I.}
Let $(\Omega_n,\theta_n,\mA_n)$ be a minimizing sequence
associated to \eqref{RelaxedProblem}. Our aim is to find a convergent subsequence (in the right sense) and prove that the limit solves the problem. 

We begin by analyzing the sequence $(\Omega_n)$. By \cite[Theorem 2.4.10]{HP18},
see also \cite{Chenais}, there exists a subsequence of $(\Omega_n)$ (for simplicity of notation, we do not relabel it),
together with an open $\Omega\subseteq D$ satisfying the $\varepsilon$-cone property such that 
$\Omega_n$ converges to $\Omega$ in the sense of Hausdorff for open sets, in the sense of 
characteristic functions and in the sense of compacts (for precise definitions, see \cite[Section 2.2]{HP18}). Moreover, both $\overline{\Omega}_n$
and $\partial\Omega_n$ converge in the Hausdorff sense for compact sets to 
$\overline{\Omega}$ and $\partial\Omega$, respectively (cf.~\cite{Holz}). Our goal is to show that the limit set $\Omega$ belongs to $\lO$.

Since all $\Omega_n$ are connected, the same holds for the sequence of closures $\overline{\Omega}_n$. Hence, it follows by 
\cite[Proposition 2.2.19]{HP18} that the limit set $\overline{\Omega}$ 
is also connected. {Consequently, since $\Omega$ is a Lipschitz domain,
it is connected as well.}
The convergence in the sense of characteristic functions implies that the volume is preserved, i.e. 
\[
|\Omega|= \|\chi_\Omega\|_{\Lj D} = \lim_n \|\chi_{\Omega_n}\|_{\Lj D}
    = \lim_n|\Omega_n| = V \,.
\]
Thus, it is left to check that $\Gamma_0\subseteq \partial\Omega$. 
This follows immediately from $\overline{\Gamma_0}\subseteq\partial\Omega_n$ and 
$\partial\Omega_n\to\partial\Omega$ in the sense of Hausdorff (see \cite[Subsection 2.2.3.2]{HP18}).

Let us now focus on $(\theta_n)$. We first extend each $\theta_n$ by zero to the whole $D$ and we denote this extension by $\tilde\theta_n$. Then we obviously have
$\tilde\theta_n\in \Lb{D;[0,1]}$. Hence, we can pass to a subsequence (again not relabeled) such that $\tilde\theta_n$ converges weakly-$*$ to 
$\tilde\theta\in \Lb{D;[0,1]}$. Let us denote $\theta:=\tilde\theta|_{\Omega}$.
Then $\theta\in \Lb{\Omega;[0,1]}$ and 
\[
\int_\Omega \theta(\mx) \,d\mx = \int_D \chi_{\Omega}(\mx) \tilde\theta(\mx)\,d\mx
    = \lim_n \int_D \chi_{\Omega_n}(\mx) \tilde\theta_n(\mx) \,d\mx
    = \lim_n \int_{\Omega_n} \theta_n(\mx) \,d\mx = q_\alpha \,.
\]

It remains to analyze the sequence $(\mA_n)$. For every $n\in\N$, 
we have $\mA_n\in \Kthn$ almost everywhere on  $\Omega_n$. Our aim is to extend $\mA_n$ to the whole domain  $D$, 
denoting the extension by $\widetilde\mA_n$,
in such a way  that the preceding property remains valid when $\mA_n$, $\theta_n$ and $\Omega_n$ 
are replaced by $\widetilde\mA_n$, $\tilde\theta_n$ and $D$, respectively:
%It is easy to see that this holds for
\[
\widetilde\mA_n(\mx) := \left\{\begin{array}{cc} 
    \mA_n(\mx), & \mx\in\Omega_n \\
    \beta\mI, & \mx\in D\setminus\Omega_n.
\end{array}\right. \,
\]

Now we can apply compactness of the H-convergence 
(see \cite[subsections 1.2.4 and 2.1.2]{All02} and \cite[Chapter 6]{Tar09}), 
which ensures existence 
of $\widetilde\mA\in\Kthtil$ and a subsequence (not relabeled) such that $(\widetilde\mA_n)$ H-converges to $\widetilde\mA$. 
Hence, for $\mA:=\widetilde\mA |_\Omega$, by the previous analysis, we have 
$(\Omega,\theta,\mA)\in{\mathcal{A}}$, i.e.~on 
the limit we got an admissible design.
%It is left to see that ${\JJ}$ attains its minimum in this admissible design, for which 
%we need to take into account \eqref{state1} as well. 

\medskip

\noindent\textbf{Step II.}
To write down weak formulation of the boundary value problem 
\begin{equation}\label{staten}
 \left\{ \begin{array}{cl} -\dv(\mA_n \nabla u^n_i)=f_i& \text{ in }\Omega_n\\ 
u^n_i=h_i &\text{ on }\Gamma_0\\
u^n_i=0 &\text{ on }\Gamma^n_1 \,, \end{array} \right.
\end{equation}
we begin by homogenizing the boundary conditions on $\Gamma_0$.

We use the same notation for the extension by zero to the whole $\R^d$ 
of $H_i$ given in (H1). 
It is well-known that $H_i\in\HH^1(\R^d)$ (see \cite[Lemma 3.27]{AF03}).
Defining $v^n_i=u^n_i-H_i|_{\Omega_n}$, for $i=1,2,\ldots, m$,
we see that $v^n_i$ is the unique solution of
\begin{equation}\label{statenh}
 \left\{ \begin{array}{cl} -\dv(\mA_n \nabla v^n_i)=f_i+g^n_i& \text{ in }\Omega_n\\ 
v^n_i=0 &\text{ on }\partial\Omega_n=\Gamma_0\cup\Gamma^n_1, \end{array} \right.
\end{equation}
where $g^n_i:=\dv{\mA}_n\nabla (H_i|_{\Omega_n})\in\Hmj{\Omega_n}$. 
This step relies on the conditions $H_i|_{\Gamma_0}=h_i$ and 
$H_i|_{\Gamma_1^n}=0$ for all $n \in \mathbb{N}$
(cf.~(H1) and Lemma \ref{lem:O}(iii)).
Applying $\mA_n\in\Kthn$, we have
\begin{equation*}%\label{eq:bound_gi}
\|g^n_i\|_{\Hmj{\Omega_n}}
	\leq\beta\|H_i\|_{\Hj{\Omega_n}}
	\leq\beta\|H_i\|_{\Hj{\R^d}}
	=\beta\|H_i\|_{\Hj{O}} \,,
\end{equation*}
where we have used that $\supp H_i\subseteq O$.

 Let us denote by 
$\widetilde\vv^n$ the extension of $\vv^n=(v_1^n,\dots,v_m^n)$ by zero on 
$D\setminus\Omega_n$. Then $\widetilde\vv^n\in \Hjn{D;\R^m}$.
Moreover, using the a priori estimate, we have that
$(\widetilde\vv^n)$ is bounded in $\Hjn{D;\R^m}$
(cf.~\cite[Proposition 3.2.1]{HP18}). Indeed, for any 
$i\in \{1,2,\dots,m\}$ and any $n\in\N$ we have
\begin{equation}\label{eq:dir_stepII_est}
\begin{aligned}
\alpha c_1 \|\widetilde v_i^n\|_{\Hjn{D}}^2 &\leq \alpha  \int_D |\nabla\widetilde v_i^n|^2 
    = \alpha  \int_{\Omega_n} |\nabla v_i^n|^2 \\
&\leq \int_{\Omega_n} \mA_n\nabla v_i^n \cdot \nabla v_i^n 
= \langle f^n_i, v_i^n\rangle_{\Hmj{\Omega_n},\Hjn{\Omega_n}}\\
%    = \langle f^n_i, \widetilde v_i^n\rangle_{\Hmj{D},\Hjn{D}} \\
&\leq c_2\|v_i^n\|_{\Hjn{\Omega_n}}=c_2\|\widetilde v_i^n\|_{\Hjn{D}} \,,
\end{aligned}
\end{equation}
where the  Poincar\'e inequality (the constant $c_1$ depends only on 
the set $D$; see \cite[Corollary 6.31]{AF03}) is applied first. In the second step we have used that $\nabla \widetilde v_i^n$ agrees 
with the zero extension of $\nabla v_i^n$ (cf.~\cite[Lemma 3.27]{AF03}). 
Then the coercivity of 
$\mA_n$ is used, the weak formulation of \eqref{statenh},
and the boundedness of $f^n_i:=f_i|_{\Omega_n}+g^n_i$ in $\Hmj{\Omega_n}$ 
($c_2$ being $\max_i (\|f_i\|_{\Ld D} + \beta\|H_i\|_{\Hj O})$).
%the fact that the actions of $f_i$ and its restriction $f_i|_{\Omega_n}$ coincide on functions in $\Hjn{\Omega_n}$ (see \cite[(3.26)]{HP18}).
%Applying Young's inequality to the right-hand side, we obtain the desired estimate.

Therefore, 
we can pass to another subsequence such that 
$(\widetilde\vv^n)$ is weakly convergent to $\widetilde\vv$ in $\Hjn{D;\R^m}$
and converges almost everywhere on $D$ (for the latter we use that by the Rellich-Kondrašov compactness theorem $(\widetilde\vv^n)$ converges strongly in $\Ld{D;\R^m}$; see \cite[Corollary 2.17 and Theorem 6.3]{AF03}). 

Finally, $\widetilde \vu^n:=\widetilde\vv_n+\vH$ converges to $\widetilde\vu:=\widetilde\vv+\vH$ weakly in $\Hj{D;\R^m}$ and almost everywhere, as well (here $\vH=(H_1,\dots,H_m)$). 
Let us introduce $\vv:=\widetilde\vv|_\Omega$ and $\vu:=\widetilde\vu|_\Omega$.
We want to show that $\vu$ is a unique solution 
of \eqref{state1}, where $\mA$ and $\Omega$ are previously defined.
\medskip

\noindent\textbf{Step III.}
For any $n\in\N$ we have $\vnul=(\chi_D-\chi_{\Omega_n})\widetilde\vv^n$ a.e.~in $D$. 
Passing to the limit (possibly on a subsequence to ensure that $(\chi_{\Omega_n})$ converges a.e.~to $\chi_\Omega$) we obtain that $\widetilde\vv=\vnul$ a.e.~on $D\setminus\Omega$. Hence, $\vv\in \Hjn{\Omega;\R^m}$ (see e.g.~\cite[Proposition 3.2.16]{HP18}).

Let us take an arbitrary $i\in\{1,2,\dots,m\}$.
It remains to show that $u_i=v_i+H_i|_\Omega$, 
satisfies the weak formulation associated with \eqref{state1}. 
By reversing the above procedure of homogenizing the boundary conditions on $\Gamma_0$, 
it is evident that it is sufficient to obtain that $v_i$ is a weak solution to 
\eqref{statenh} with $\Omega_n$, $\mA_n$ and $g_i^n$
replaced by $\Omega$, $\mA$ and $g_i:=\dv \mA\nabla (H_i|_\Omega)$, respectively. 
Note that $g_i$ satisfies the same bound as $g_i^n$.

Fix an arbitrary test function 
$\varphi \in \Cbc\Omega$. Given this test function $\varphi$, 
select an open set $U$ such that $\supp\varphi \subseteq U \ssubset \Omega$.
Since $(\Omega_n)$ converges to $\Omega$ in the sense of compacts
(cf.~\cite[Proposition 2.2.17]{HP18}), there exists $n_1=n_1(U)\in\N$
such that for any $n\geq n_1$ we have $U\ssubset\Omega_n$.
In particular, $\varphi\in\Cbc{\Omega_n}$.
Thus, $\varphi$ is an admissible test function for the weak formulation of the 
problem for $v^n_i$, i.e.~for any $n\geq n_1$ we have
\begin{equation*}%\label{eq:weak_form_U}
\int_{\Omega_n} \mA_n\nabla v_i^n \cdot\nabla \varphi \,d\mx 
    =  \langle f_i|_{\Omega_n}+ g_i^n ,\varphi\rangle_{\Hmj{\Omega_n},\Hjn{\Omega_n}} \,.
\end{equation*}
By moving $\langle  g_i^n ,\varphi\rangle_{\Hmj{\Omega_n},\Hjn{\Omega_n}} = -\int_{\Omega_n} \mA_n \nabla (H_i|_{\Omega_n})\cdot\nabla\varphi  \,d\mx$ to the left-hand side we get $u_i^n$ in place 
of $v_i^n$. Moreover, using that $f_i\in\Ld D$ and
$\supp \varphi\subseteq U\ssubset \Omega\cap\Omega_n$, we have
\begin{equation*}%\label{eq:weak_form_rhs}
\begin{aligned}
\langle f_i|_{\Omega_n},\varphi\rangle_{\Hmj{\Omega_n},\Hjn{\Omega_n}}
	= \int_U f_i(\mx)\varphi(\mx) \,d\mx   \,. 
\end{aligned}
\end{equation*}
Thus, the following holds:
\begin{equation}\label{eq:weak_form_uin}
\int_{U} \mA_n\nabla u_i^n \cdot\nabla \varphi \,d\mx 
	=  \int_U f_i(\mx)\varphi(\mx) \,d\mx \,,
\end{equation}
where we accounted for the fact that $\varphi \in\Cbc{U}$.

It remains to analyse the limit of the left-hand side in \eqref{eq:weak_form_uin}. 
More precisely, we want to show that it converges, 
as $n$ tends to infinity, to $\int_U \mA\nabla u_i\cdot\nabla\varphi\,d\mx$.
Here, we utilize the concept of H-convergence to study the behavior of the term 
$\mA_n\nabla u_i^n$, with the aim of applying \cite[Proposition 1.2.19]{All02}
(see also \cite[Lemma 13.2]{Tar09}).

Let us start by noting that, as shown in \cite[Lemma 10.5]{Tar09},
$\mA_n|_U=\widetilde\mA_n|_U$ H-converges to $\mA|_U=\widetilde\mA|_U$
(recall that $n\geq n_1$).
Furthermore, since \eqref{eq:weak_form_uin} holds 
for any $\varphi\in\Cbc U$, we have that 
\begin{equation*}
-\dv\bigl(\mA_n|_U \nabla (u_i^n)|_U\bigr)  = f_i|_U\in \Ld{U} \,.
\end{equation*}
%Finally, from the weak convergence $\widetilde u_i^n\rightharpoonup \widetilde u_i$ in
%$\Hjn{D}$, one easily gets that $v_i^n|_U\rightharpoonup v_i|_U$ weakly in
%$\Hj{U}$. 
%(it is sufficient to consider an arbitrary $\psi\in\Cbc U$ and  observe that still $\psi \widetilde u_i^n\rightharpoonup\psi\widetilde u_i$ weakly in $\Hjn{D}$).
Therefore, since $u_i^n|_U\rightharpoonup u_i|_U$ weakly in
$\Hj{U}$, we may apply \cite[Proposition 1.2.19]{All02}, which yields 
$\mA_n|_U\nabla (u_i^n)|_U\rightharpoonup \mA|_U\nabla (u_i)|_U$ weakly
in $\Ldl{U;\R^d}$. Returning to the left-hand side of \eqref{eq:weak_form_uin} we get
\begin{equation*}
\begin{aligned}
\lim_{n\to\infty}\int_{U} \mA_n\nabla u_i^n \cdot\nabla \varphi \,d\mx
	&= \int_{U} \mA\nabla u_i \cdot\nabla \varphi \,d\mx 
	= \int_{\Omega} \mA\nabla v_i \cdot\nabla \varphi \,d\mx 
		- \langle g_i,\varphi\rangle_{\Hmj{\Omega},\Hjn{\Omega}} \,.
\end{aligned}
\end{equation*}

Finally, by the arbitrariness of $\varphi$ and  the density of $\Cbc\Omega$ in $\Hjn\Omega$,
the function $v_i$ is the weak solution of \eqref{statenh}, with $\Omega_n$, $\mA_n$ and $g_i^n$
replaced by $\Omega$, $\mA$ and $g_i$, respectively.
In light of the preceding observations, this immediately leads  the 
desired conclusion  that $u_i$ is the weak solution to \eqref{state1}.

%This demonstrates an important property: the sequence of zero extensions
%of the unique solutions $(\widetilde\vu_n)$ associated to 
%$(\Omega_n,\theta_n,\mA_n)$ and $\vf$
%converges weakly to the zero extension $\widetilde\vu$
%of the unique solution corresponding to
%$(\Omega,\theta,\mA)$ and $\vf$.
%The weak convergence $\widetilde\vu^n\rightharpoonup \widetilde\vu$ in $\Hjn{D}$ is a crucial step in establishing the existence of minimizers for $\overline{\JJ}$, as will be shown in the next step.

\medskip

\noindent\textbf{Step IV.} 
The final task is to prove that ${\JJ}(\Omega_n,\theta_n,\mA_n)$
converges to ${\JJ}(\Omega,\theta,\mA)$, i.e.~that the relaxed admissible
design $(\Omega,\theta,\mA)$ is the minimizer of ${\JJ}$ on 
${\mathcal{A}}$.

The value of the functional ${\JJ}$ at the design configuration
$(\Omega_n,\theta_n,\mA_n)$ can be expressed as (note that here we work
with the final subsequence obtained in the previous step):
\begin{equation*}
{\JJ}(\Omega_n,\theta_n,\mA_n)
    =\int_D \chi_{\Omega_n}(\mx) \Bigl( \tilde\theta_n(\mx) g_{\alpha}(\mx ,\widetilde\vu_n(\mx)) + (1-\tilde\theta_n(\mx)) g_{\beta}(\mx ,\widetilde\vu_n(\mx )) \Bigr) \,\mathrm{d}\mx \,.
\end{equation*}
This representation is useful for clearly separating the ingredients that are already 
under control from those that still require justification in order to pass to the 
limit as $n$ tends to infinity. 
Indeed, $\chi_{\Omega_n}$ strongly converges to $\chi_\Omega$
in any $\pL sD$, $s\in [1,\infty\rangle$ (since $(\Omega_n)$
converges to $\Omega$ in the sense of characteristic functions;
cf.~\cite[Definition 2.2.3]{HP18}), while $(\tilde\theta_n)$
converges weakly-$*$ to $\tilde\theta$ in $\pL \infty D$.
Consequently, it suffices to prove that there exists $p>1$
such that
$g_\gamma(\cdot,\widetilde\vu_n)$ converges strongly to 
$g_\gamma(\cdot,\widetilde\vu)$ in $\pL pD$, $\gamma\in\{\alpha,\beta\}$.

Since $\widetilde\vu_n(\mx)$ converges to $\widetilde\vu(\mx)$
for a.e.~$\mx\in D$ and $g_\gamma$, $\gamma\in\{\alpha,\beta\}$, are 
Carath\' eodory functions on $D\times\R^m$, we have that
\begin{equation*}
g_\gamma(\mx,\widetilde\vu_n(\mx)) \overset{n\to\infty}{\longrightarrow} g_\gamma(\mx,\widetilde\vu(\mx)) \quad \hbox{a.e.~$\mx\in D$},
    \gamma\in\{\alpha,\beta\}\,.
\end{equation*}
In order to conclude the argument via the Lebesgue dominated convergence theorem, it remains to find, for some $p>1$, 
an $\mathrm{L}^p$-integrable function that serves as a dominating
function for $|g_\gamma(\cdot,\widetilde\vu_n)|$.

A more precise application of the Rellich-Kondrašov theorem 
(see \cite[Theorem 6.3]{AF03}) implies that, for any
$r\in [1, q^\ast\rangle$, the sequence
$(\widetilde\vu_n)$ converges strongly to $\widetilde\vu$ in $\pL rD$, 
where $q^\ast$ is defined in (A5).
Hence, for any $r\in [1, q^\ast\rangle$ there exists 
$w\in \pL rD$ such that %$|\widetilde\vu_n(\mx)|\geq w(\mx)$ a.e.~$\mx\in D$. 
$|\widetilde\vu_n(\mx)|\leq w(\mx)$ a.e.~$\mx\in D$.
Then from the growth condition on $g_\gamma$ we get
\[
|g_\gamma(\mx,\widetilde\vu_n(\mx))| \leq |\varphi_\gamma(\mx)|
    + |\psi_\gamma(\mx)| w(\mx)^q =: G(\mx) \,, \quad \hbox{a.e.~$\mx\in D$} \,.
\]
Now $w^q\in\pL {\frac rq}D$ and by the H\"older inequality $\psi_\gamma w^q$ belongs to $\pL sD$ with $s>1$ if 
\[
\frac1s:=\frac 1{s_2}+ \frac qr<1, 
%\qquad \Longleftrightarrow \qquad \frac 1{s_2} =\frac1s -  \frac qr< 1 -  \frac q{q^\ast}=\frac{q^\ast-q}{q^\ast},
\]
where we have the freedom to choose the value $r\in[1,q^\ast\rangle$. For $d=2$ by (A5) we have
$\frac1{s_2}<1$, and since $q^\ast=\infty$, the parameter $r$ can be chosen so that $\frac qr$ is arbitrarily small. On the other hand, for $d>2$ it suffices that $\frac 1{s_2}+ \frac q{q^\ast}<1$ because $r$ can be taken arbitrarily close to $q^\ast$. Hence the assumption (A5) ensures that these conditions are satisfied.
Therefore, $G\in\pL pD$, with $p=\min\{s_1,s\}>1$. 
This completes the proof. 
%, where $s$ satisfies the equality above. 

\end{proof}

\begin{remark}
The result of Theorem \ref{tm:dir} still holds if we remove 
assumption {\rm (A3)} and set $\Gamma_0=\emptyset$. 
However, from the perspective of concrete applications, 
this scenario is not of particular importance, as the position 
of the admissible sets is typically \emph{fixed} using the 
fixed portion of the boundary $\Gamma_0$.
\end{remark}

\subsection{Neumann-type free boundary}

To begin with, we consider the case in which the entire free part of the boundary is 
subject to a homogeneous Neumann boundary condition, that is $\Gamma_1=\emptyset$. 
We impose the following assumption on the functions $h_i$:
\begin{itemize}
	\item[(H2)] For each $i\in\{1,2,\dots, m\}$ there exists $H_i\in \HH^1(D)$ such 
    that $h_i=H_i|_{\Gamma_0}$.
\end{itemize}
%For a detailed discussion on the assumptions imposed on $h_i$, we refer to  Remark \ref{rem:assum_hi}. 

By (A3) we know that the manifold boundary of $M$, $\partial M$, 
is a compact $(d-2)$-dimensional Lipschitz manifold 
(without boundary). Hence, $\partial M$ is locally the graph of a 
Lipschitz-continuous function and has 
a finite number of connected components.
This is important since for any admissible domain $\Omega\in\mathcal{O}$
we have $\overline{\Gamma_0}\cap \Gamma_2 = \partial M$, 
where $\Gamma_2=\partial\Omega\setminus \Gamma_0$. Hence, 
the set of all smooth functions $\varphi\in{\rm C}^\infty(\overline{D})$ 
that vanish on a neighborhood of $\Gamma_0$, 
when restricted to $\Omega$, is dense in 
\begin{equation}\label{eq:neum_V}
V:=\{v\in\Hj{\Omega}:{v}|_{\Gamma_0}=0\}
\end{equation}
(see \cite{Bern,BLD}).

We are now prepared to state and prove the following theorem.

%For given $\Omega$ and conductivity $\mA$ the problem \eqref{state1}, provided that each $f_i$ belongs to $V'$, admits a unique solution $u_i\in w_i+V$, where $w_i\in\Hj{\Omega}$ is a function whose trace on $\Gamma_0$ equals $h_i$, for each $i=1,\ldots, m$. 
%For the purposes of the present analysis, we shall assume that $\vf\in\pL2{D;\R^m}$, with the understanding that this assumption will be further restricted in the subsequent numerical analysis. 

\begin{theorem}\label{tm:neu}
%\mbox{}\\
  Under assumptions {\rm (A1)}--{\rm (A5)} and {\rm (H2)}, there exists a solution to
  the optimal design 
  problem \eqref{RelaxedProblem} when $\Gamma_1=\emptyset$.
\end{theorem}

\begin{proof}
Following the strategy used in the proof of Theorem \ref{tm:dir}, 
we shall highlight only the steps where noteworthy differences 
are present.

As in the proof of Theorem \ref{tm:dir}, we begin with a minimizing sequence 
$(\Omega_n,\theta_n, \mA_n)$, and Step I from the previous proof applies unchanged.

The following two steps requires greater care. 
Let us start with the analysis of the boundary value problem (Step II in the proof of Theorem \ref{tm:dir}).
For any $i\in\{1,2,\dots,m\}$ and any $n\in\N$, 
there is a unique solution $v_i^n$ to the following variational problem:
\begin{equation}\label{eq:weakneun}
\int_{\Omega_n} \mA_n\nabla (v^n_i+H_i)\cdot\nabla\varphi\,d\mx=\int_{\Omega_n} f_i\varphi\,d\mx\,,\;\varphi\in V_n\,,
\end{equation}
where $H_i$ is given in (H2) and
\[
V_n:=\{v\in\Hj{\Omega_n}:v|_{\Gamma_0}=0\} \,.
\]
Then it is easy to check that $u_i^n = H_i|_{\Omega_n} + v_i^n \in H_i|_{\Omega_n} + V_n$
is the unique weak solution to  
\begin{equation}\label{eq:neumann_u_n}
\left\{ \begin{array}{cc} -\dv(\mA_n \nabla u_i^n)=f_i& \text{ in }\Omega_n\\ 
u_i^n=h_i &\text{ on }\Gamma_0\\
\mA\nabla u_i^n\cdot \mn=0 &\text{ on }\Gamma_2^n = \partial\Omega_n\setminus\Gamma_0. \end{array} \right.
\end{equation}

In contrast to the previous proof, we make use of the extension $\widetilde\vv^n$ to the set $D$ constructed according to the method presented by Chenais \cite{Chenais}. Due to the uniform cone condition, there exists a constant $c_0>0$, not depending on $n$ such that
\begin{equation}\label{eq:chenais}
\|\widetilde v^n_i\|^2_{\Hj{D}}\leq c_0\|v^n_i\|^2_{\Hj{\Omega_n}} \,.
\end{equation}
Another crucial step is the usage of the Poincar\'e inequality. 
While in the previous theorem the situation was trivial (it was applied on 
the space $\Hjn D$), here the existence of the uniform estimate for the 
extensions ensures that the constant in the Poincar\'e inequality does not 
depend on $n$ (see Lemma \ref{lem:poincare} below). Thus, if we denote the constant by $\sqrt{c_1}$, we have
\begin{equation}\label{eq:ocj}
\begin{aligned}
\alpha c_1 \|\widetilde v^n_i\|^2_{\Hj{D}}&\leq \alpha c_1 c_0\|v^n_i\|^2_{\Hj{\Omega_n}} \\
&\leq  c_0 \int_{\Omega_n} \mA_n\nabla v^n_i\cdot\nabla v^n_i \,d\mx
	= c_0 \int_{\Omega_n}      f_i v^n_i -\mA_n\nabla H_i\cdot\nabla v^n_i\,d\mx\\
&\leq  c_0 \|f_i\|_{\Ld{D}}\|\widetilde v_i^n\|_{\Ld{D}} 
	+ c_0 \beta \|\nabla H_i\|_{\Ld{D;\R^d}}\| 
	\nabla \widetilde v_i^n\|_{\Ld{D;\R^d}}\\
&\leq c_2\|\widetilde v_i^n\|_{\Hj{D}}\,, 
\end{aligned}
\end{equation}
where $c_2$ does not depend on $n$. 
Hence, the boundedness of the sequence $(\widetilde\vv^n)$ in $\Hj{D;\R^m}$ follows. Therefore, now we can easily adapt Step II of the previous proof and obtain its weak limit $\widetilde\vv\in\Hj{D;\R^m}$ (at a subsequence), and denote  $\vv:=\widetilde\vv|_\Omega\in\Hj{\Omega;\R^m}$. Since the trace operator is linear and continuous, it follows that 
$\vv=\vnul$ on $\Gamma_0$, i.e.~each component of $\vv$ belongs to $V$ (see \eqref{eq:neum_V}). At the same subsequence we have weak convergence of $\widetilde\vu^n:=\widetilde\vv^n+\vH$ to $\widetilde\vu:=\widetilde\vv+\vH$, and in the same manner we denote  $\vu:=\widetilde\vu|_\Omega$, belonging to the desired space $\vH+V^m$.

The next step is to pass to the limit in \eqref{eq:weakneun} (and then also in \eqref{eq:neumann_u_n}), which is more delicate when compared to the procedure 
used in the previous theorem. 

For $\eta>0$ we introduce
\[
\begin{aligned}
    \omega_\eta&=\{\mx\in \Omega:d(\mx;\partial\Omega)\geq\eta\}\\
    \Gamma_\eta&=\{\mx\in \R^d:d(\mx;\partial\Omega)<\eta\}.
\end{aligned}
\]

The compact set $\omega_\eta$ is included in $\Omega$ by definition, 
and for any positive $\eta$ we have $\overline{\Omega}\subseteq\Gamma_\eta\,\dot\cup\,\omega_\eta$.
Therefore, since $(\Omega_n)$ converges to $\Omega$ in the sense of 
compacts, for $n$ large enough we have
\[
\begin{aligned}
    \omega_\eta&\ssubset\Omega_n\\
   \Omega_n&\subseteq (\Gamma_\eta \,\dot\cup\,\omega_\eta)\cap D
   	= (\Gamma_\eta\cap D)\,\dot{\cup}\, \omega_\eta \,,
\end{aligned}
\]
where we in addition used that $\Omega_n\subseteq D$. 

We aim to pass to the limit in the equality \eqref{eq:weakneun}, which can be equivalently rewritten as follows:
\begin{equation}\label{eq:weakneuneq}
\int_{D} \chi_{\Omega_n}\widetilde\mA_n\nabla \widetilde u^n_i\cdot\nabla\varphi\,d\mx=\int_{D} \chi_{\Omega_n} f_i\varphi\,d\mx \,, %\,,\;\varphi\in V_n\,.
\end{equation}
for any $\varphi\in{\rm C}^\infty(\overline{D})$ that vanish on a neighborhood of $\Gamma_0$. Here we make use of the fact that the set of all such functions, when restricted to $\Omega_n$ is dense in $V_n$ \cite[Theorem 3.1]{Bern}.

The right-hand side of the previous equality converges to $\int_{D} \chi_{\Omega} f\varphi_i\,d\mx$, since $\Omega_n$ converges to $\Omega$ in the sense of characteristic functions.

Furthermore, the integral on the left-hand side of \eqref{eq:weakneuneq} is the sum of the same integrands over $\omega_\eta$ and $\Gamma_\eta\cap D$.
Again by \cite[Proposition 1.2.19]{All02} we have the convergence
$\mA_n|_{\omega_\eta}\nabla (u_i^n)|_{\omega_\eta}\rightharpoonup \mA|_{\omega_\eta}\nabla (u_i)|_{\omega_\eta}$ weakly
in $\Ld{\omega_\eta;\R^d}$ which implies the convergence
\[
\begin{aligned}
\int_{\omega_\eta} \chi_{\Omega_n}\mA_n\nabla u^n_i\cdot\nabla\varphi\,d\mx &=\int_{\omega_\eta}\mA_n\nabla u^n_i\cdot\nabla\varphi\,d\mx\to
\int_{\omega_\eta} \mA\nabla u_i\cdot\nabla\varphi\,d\mx\\
&=\int_{\omega_\eta} \chi_{\Omega}\mA\nabla u_i\cdot\nabla\varphi\,d\mx\,.
\end{aligned}
\]
Finally, for the integral over $\Gamma_\eta$ we have
\[
\begin{aligned}
\Biggl|\int_{\Gamma_\eta\cap D} & \chi_{\Omega_n}\widetilde\mA_n\nabla\widetilde u^n_i\cdot\nabla\varphi\,d\mx -\int_{\Gamma_\eta\cap D} \chi_{\Omega}\widetilde\mA\nabla \widetilde u_i\cdot\nabla\varphi\,d\mx
\Biggr|\\
&\leq\int_{\Gamma_\eta\cap D} |\widetilde\mA_n\nabla \widetilde u^n_i\cdot\nabla\varphi|\,d\mx +\int_{\Gamma_\eta\cap D} |\widetilde\mA\nabla \widetilde u_i\cdot\nabla\varphi|\,d\mx\\
&\leq\|\widetilde\mA_n\nabla \widetilde u^n_i\|_{\Ld{D;\R^d}}\|\nabla \varphi\|_{\Ld{\Gamma_\eta\cap D;\R^d}}+
\|\widetilde\mA\nabla \widetilde u_i\|_{\Ld{D;\R^d}}\|\nabla \varphi\|_{\Ld{\Gamma_\eta\cap D;\R^d}}\\
&\leq \beta \|\nabla\varphi\|_{\Lb{D;\R^d}}\bigl(\|\widetilde u_i^n\|_{\Hj D}
	+\|\widetilde u_i\|_{\Hj D}\bigr) \sqrt{|\Gamma_\eta\cap D|} \\
&\leq c\sqrt{|\Gamma_\eta|}\,,
\end{aligned}
\]
where constant $c$ does not depend on $n$. 
Since $|\Gamma_\eta|$ converges to zero as $\eta$ goes to zero, 
we can pass to the limit in \eqref{eq:weakneuneq}  and conclude
\begin{equation*}
\int_{\Omega} \mA\nabla u_i\cdot\nabla\varphi\,d\mx=\int_{\Omega}  f_i\varphi\,d\mx\,,\ i=1,\ldots,m\,,
\end{equation*}
for any $\varphi\in{\rm C}^\infty(\overline{D})$ that vanish on a neighborhood of $\Gamma_0$, which means that the limit $\vu$ solves \eqref{state1}.

Finally, we apply Step IV of the previous proof directly, which completes the proof.
\end{proof}

Here we provide an important uniform Poincaré inequality for admissible domains. 
The strategy of the proof is standard and relies on Chenais's extension 
(see e.g.~\cite[Lemma 2.19]{HM03}); 
however, for completeness, we provide the full proof.

\begin{lemma}\label{lem:poincare}
There exists a constant $c>0$ such that for any admissible domain
$\Omega\in \mathcal{O}$ and $u\in V=\{v\in\Hj{\Omega}:{v}|_{\Gamma_0}=0\}$ it holds
$$
c\,\|u\|_{\Hj\Omega}\leq \|\nabla u\|_{\Ld{\Omega;\R^d}} \,.
$$
\end{lemma}
\begin{proof}
Let us assume that the claim does not hold. 
Then for any $n\in\N$ there exist $\Omega_n\in\mathcal{O}$ and 
$u_n\in V_n = \{v\in\Hj{\Omega_n}:{v}|_{\Gamma_0}=0\}$ such that 
$\|u_n\|_{\Ld{\Omega_n}}=1$ and
$$
\frac{1}{n} \|u_n\|_{\Hj{\Omega_n}} > \|\nabla u_n\|_{\Ld{\Omega_n;\R^d}} \,.
$$
Hence, since $\|u_n\|_{\Hj{\Omega_n}}^2=1 + \|\nabla u_n\|_{\Ld{\Omega_n;\R^d}}^2$,
we get 
$\|\nabla u_n\|_{\Ld{\Omega_n;\R^d}} < \frac{1}{\sqrt{n^2-1}}$, $n\geq 2$,
implying 
\begin{equation*}
\lim_{n\to\infty} \|\nabla u_n\|_{\Ld{\Omega_n;\R^d}} =0 \,.
\end{equation*}
In particular, the sequence of real numbers $(\|u_n\|_{\Hj{\Omega_n}})$ is bounded. 

Let us denote by $\widetilde u_n$ the extension of $u_n$ following the method
presented by Chenais \cite{Chenais}. 
Due to the uniform cone condition, there exists a constant 
$c_0>0$ independent of $n$ such that 
\begin{equation*}
\|\widetilde u_n\|_{\Hj{D}} \leq c_0 \|u_n\|_{\Hj{\Omega_n}} \,, \quad n\in\N \,.
\end{equation*}
Thus, $(\widetilde u_n)$ is bounded in $\Hj{D}$.
Let us pass to a subsequence (keeping the same notation) such that 
$(\widetilde u_n)$ converges to $v$ weakly in $\Hj{D}$ and 
strongly in ${\rm L}^q(D)$ for any $q\in [1,q^\ast)$ 
(see \cite[Theorem 6.3]{AF03}), 
where $q^\ast$ is defined in (A5).

According to \cite[Theorem 2.4.10]{HP18}
(see also \cite{Chenais}), we can pass to another subsequence such that
$\Omega_n$ converges to $\Omega\in\mathcal{O}$ 
in the sense of Hausdorff (for open sets) and
in the sense of characteristic functions.

Let us take an arbitrary $\varphi\in\Cbc{D}$. 
Using the Cauchy-Schwarz inequality, it holds 
\begin{equation*}
\left|\int_{\Omega_n} \varphi(\mx)  \nabla u_n(\mx) \, d\mx\right|
\leq \|\nabla u_n\|_{\Ld{\Omega_n;\R^d}} \|\varphi\|_{\Ld{D}}
\xrightarrow{n\to\infty} 0 \,.
\end{equation*}
On the other hand, since $(\chi_{\Omega_n})$ converges strongly to $\chi_\Omega$ in 
$\Ld{D}$ and $(\nabla \widetilde u_n)$ converges weakly to 
$\nabla v$ in $\Ld{D;\R^d}$, we have
\begin{equation*}
\begin{aligned}
0=\lim_n \int_{\Omega_n} \varphi(\mx)  \nabla u_n(\mx) \, d\mx
	&= \lim_n \int_D \varphi(\mx) \chi_{\Omega_n}(\mx) \nabla\widetilde u_n(\mx) \, d\mx \\
&=  \int_D\varphi(\mx) \chi_\Omega(\mx) \nabla v(\mx) \,d\mx \\
&=  \int_\Omega \varphi(\mx) \nabla v(\mx) \,d\mx \,.
\end{aligned}
\end{equation*}
Since $\varphi\in\Cbc{D}$ was arbitrary, we get that $\nabla v=\vnul$ a.e.~in $\Omega$. 
Thus, $v$ is a constant function on $\Omega$ (recall that $\Omega$ is connected).

Further on, since the trace is weakly sequentially continuous, 
from $\widetilde u_n|_{\Gamma_0} = u_n|_{\Gamma_0}=0$, follows $v|_{\Gamma_0}=0$. 
Thus, $v=0$ a.e.~in $\Omega$.

However, since $(\widetilde u_n^2)$ converges strongly in ${\rm L}^p(D)$ for some $p>1$
(in fact, $p$ is any number in the interval $(1,\frac{q^\ast}{2})$)
and $(\chi_{\Omega_n})$ converges strongly in ${\rm L}^r(D)$ for any 
$r\in [1,\infty)$, we can pass to the limit:
\begin{equation*}
1= \|u_n\|_{\Ld{\Omega_n}}^2 = \int_D \chi_{\Omega_n}(\mx) \bigl(\widetilde u_n(\mx)\bigr)^2
	\, d\mx \xrightarrow{n\to\infty} \int_\Omega v(\mx)^2 \,d\mx =0 \,,
\end{equation*}
where we have used that $v=0$ a.e.~in $\Omega$.

Therefore, we have obtained a contradiction,
which ensures that the statement of the lemma must hold.
\end{proof}

\begin{remark}
The volume constraint $|\Omega|=V$ present in the definition  of the set of admissible domains $\mathcal{O}$ (see \eqref{assump1}) is irrelevant to the assertion of the previous lemma and can be omitted.
\end{remark}

\subsection{General boundary conditions}

Let us now consider the general case, where both $\Gamma_1$ and $\Gamma_2$ are (possibly) nonempty. 
For potential practical implementations, we naturally introduce subsets $Q_1$ and $Q_2$ of the domain $D$ which determine where the Dirichlet or Neumann boundary conditions are applied.
More precisely, for a given admissible domain $\Omega$, we will define
$\Gamma_i = \partial\Omega\cap Q_i$, $i=1,2$.
 Nevertheless, as the subsequent example demonstrates, additional assumptions concerning the boundary sets $\Gamma_1$ and $\Gamma_2$ (i.e.~on the family of admissible domains) are necessary.

\begin{example}\label{ex:gen}

    Let $D=\langle-2,2\rangle\times \langle-2,2\rangle$, with $Q_1=[-2,2]\times\langle0,2]$ and $Q_2=[-2,2]\times[-2,0]$, and $\Gamma_0=[0,1]\times\{1\}$.
We introduce a sequence of boundary value problems on 
$\Omega_n=\langle0,1\rangle\times\langle1/n,1\rangle$:
    \begin{equation}\label{state2}
 \left\{ \begin{array}{cc} -\Delta u^n=0& \text{ in }\Omega_n\\ 
u^n(x,y)=\sin \pi x &\text{ on }\Gamma_0\\
u^n=0 &\text{ on }\Gamma_1^n=\partial\Omega_n\setminus\Gamma_0.
 \end{array} \right.
\end{equation}

The sequence of domains $(\Omega_n)$ converges to $\Omega = \langle 0,1 \rangle \times \langle 0,1 \rangle$ in the Hausdorff sense, which implies that the boundary condition on the lower side is altered. The limiting boundary value problem states:
    \begin{equation}\label{state3}
 \left\{ \begin{array}{cc} -\Delta u=0& \text{ in }\Omega\\ 
u(x,y)=\sin \pi x &\text{ on }\Gamma_0\\
u=0 &\text{ on }\Gamma_1=\left(\{0\}\times\langle0,1]\right)\cup\left(\{1\}\times\langle0,1]\right)\\
\nabla u\cdot\mn =0 &\text{ on }\Gamma_2=[0,1]\times\{0\}.
 \end{array} \right.
\end{equation}

It is straightforward to solve the problem using the Fourier separation method. One observes that the unique solution of \eqref{state2} is 
\[
u^n(x,y)=\frac{\sh\pi\left(y-\frac1n\right)}{\sh\pi\left(1-\frac1n\right)}\sin\pi x,
\]
converging uniformly (on any compact subset of $\mathbb{R}^2$) to
 $\displaystyle\frac{\sh\pi y}{\sh\pi}\sin\pi x$, whereas the unique solution of \eqref{state3} is
\[
u(x,y)=\displaystyle\frac{\ch\pi y}{\ch\pi}\sin\pi x.
\]
Hence, it is impossible for any extension of $u^n$ to converge weakly in $\HH^1$ to an 
extension of $u$, which is essential
for the validity of the preceding proofs.
\end{example}

To avoid unwanted phenomena, we narrow the set of admissible domains by introducing an additional assumption:
\begin{itemize}
\item[(A6)] $Q_1,Q_2\subseteq \overline{D}$ are given disjoint compact sets such that $\overline{\Gamma_0}\setminus(Q_1\cup Q_2) = \Gamma_0$, where $\Gamma_0$ is given in (A3). 
\end{itemize}
This assumption ensures a clear separation between the variable
parts of the Dirichlet and Neumann boundary conditions, 
thereby avoiding the pathological scenario described in  Example \ref{ex:gen}.  
Furthermore, the condition $\overline{\Gamma_0}\setminus(Q_1\cup Q_2) = \Gamma_0$ 
implies that the boundary of the fixed region, 
$\partial\Gamma_0$, is entirely contained in $\partial Q_1\cup \partial Q_2$.
%Finally, it is important to notice that it is allowed for both $Q_1$ and $Q_2$ to be empty sets (see Remark \ref{rem:Q1Q2}).
Finally, note that $Q_1$ and $Q_2$ may be empty, but not simultaneously (see Remark \ref{rem:Q1Q2}).

Since we will use that $\Gamma_i=\partial\Omega\cap Q_i$, $i=1,2$, 
a necessary condition is $\partial\Omega\setminus (Q_1\cup Q_2)=\Gamma_0$.
This condition we need to include in the set of admissible designs. Hence, the new sets of \emph{admissible domains} and \emph{designs} 
reads, respectively,
\begin{equation*}
\begin{aligned}
\mathcal{O}' = \Bigl\{\Omega\subseteq D : \Omega\subseteq D & \text{ is open, connected and satisfies the $\varepsilon$-cone property}, \\
& |\Omega|=V, \; \Gamma_0\subset \partial\Omega, \;
\partial\Omega\setminus (Q_1\cup Q_2)= \Gamma_0 \Bigr\}
\end{aligned}
\end{equation*}
and
\begin{equation}\label{eq:addmis_mixed}
\begin{aligned}
\mathcal{A'} =\Bigl\{(\Omega,\theta,\mA): \Omega\in \mathcal{O}', \; \theta\in\Lb{\Omega;[0,1]}, \; \int_\Omega \theta\, d\mx =q_\alpha, \; \mA\in \Kth \: \text{ a.e. in }\Omega\Bigr\} \,.
\end{aligned}
\end{equation}
In order to ensure that the set of admissible designs is non-empty, 
we need to strengthen assumption (A4):
\begin{itemize}
\item[(A4$'$)] There exists $\widehat\Omega\in\mathcal{O}'$. 
\end{itemize}

Finally, we impose the following assumption on functions $h_i$:
\begin{itemize}
	\item[(H3)] For any $i\in\{1,2,\dots, m\}$ there exists $H_i\in \HH^1(D)$ such that $h_i=H_i|_{\Gamma_0}$ and $H_i=0$ a.e.~in $Q_1$.
\end{itemize}

For a given admissible design $(\Omega,\theta,\mA)\in\mathcal{A}'$, 
let $\vu$ denote the unique solution of \eqref{state1}, 
where  the boundary segments are defined by $\Gamma_i=\partial\Omega\cap Q_i$ for
$i=1,2$.
By this definition of $\Gamma_i$, $i=1,2$, and the definition of $\mathcal{O}'$ we indeed have
$$
\begin{aligned}
\partial\Omega &= (\partial\Omega\cap Q_1) \,\dot{\cup}\, (\partial\Omega\cap Q_2)
	\,\dot{\cup}\, \bigl(\partial\Omega\setminus(Q_1\cup Q_2)\bigr) \\
&= \Gamma_1 \,\dot{\cup}\, \Gamma_2 \,\dot{\cup}\, \Gamma_0 \,.
\end{aligned}
$$
Hence, we are interested in solving the optimal design problem
\eqref{RelaxedProblem}, with $\mathcal{A}$  replaced by $\mathcal{A}'$.

We have the following result.
\begin{theorem}\label{tm_gen}
%\mbox{}\\
Under assumptions {\rm (A1)--(A3)}, {\rm (A4$'$)}, {\rm (A5), (A6)},
and {\rm (H3)},
the optimal design problem \eqref{RelaxedProblem}
admits a solution, where the set of admissible designs $\mathcal{A}$
is replaced by $\mathcal{A}'$ (cf. \eqref{eq:addmis_mixed}) and, 
in \eqref{state1}, the boundary segments are defined by $\Gamma_i=\partial\Omega\cap Q_i$ for $i=1,2$.
\end{theorem}

\begin{remark}\label{rem:Q1Q2}
This final general setting encompasses both previously addressed ones. 
Here we rely on Lemma \ref{lem:O}.

More precisely, for the Dirichlet problem discussed in Theorem \ref{tm:dir},
we set $Q_1=\overline{D}\setminus O$ and $Q_2=\emptyset$.
Under this choice, assumption {\rm (A6)} is automatically satisfied by
Lemma \ref{lem:O}, as is the additional condition
$$
\partial\Omega\setminus(Q_1\cup Q_1)=\partial\Omega\cap O = \Gamma_0 \,.
$$
Hence, we have $\mathcal{O}'=\mathcal{O}$ and
$\mathcal{A}'=\mathcal{A}$.

For the mixed-problem with the fixed Dirichlet part studied in Theorem \ref{tm:neu}
we just take the reverse choice: $Q_1=\emptyset$ and $Q_2=\overline{D}\setminus O$.
Therefore, Theorem \ref{tm_gen} is indeed more general when compared to 
Theorems \ref{tm:dir} and \ref{tm:neu}.

In terms of admissible domains and their corresponding boundary conditions, the setting of Theorem \ref{tm_gen} is quite general. In particular, it covers problems of Bernoulli type \cite{HKKP03} as well as pipe-like (tubular) domains \cite{Privat10}. 
However, at this level of generality, we cannot allow the moving portions of the Dirichlet and Neumann boundaries to meet (see Example \ref{ex:gen}). To accommodate such cases, more restrictive assumptions on the admissible domains must be imposed, as seen in \cite[Chapter 2]{HM03}.
\end{remark}

\begin{proof}[Proof of Theorem \ref{tm_gen}]

We employ the same steps used in the previous proofs. 

Since $\mathcal{A}'\subseteq \mathcal{A}$, for a minimizing sequence 
$(\Omega_n,\theta_n,\mA_n)\in\mathcal{A}'$ following the construction from Step I 
we obtain $(\Omega,\theta,\mA)\in\mathcal{A}$.
Hence, it is left to check that $\partial\Omega\setminus(Q_1\cup Q_2)=\Gamma_0$. 

The inclusion $\Gamma_0\subseteq\partial\Omega$ is guaranteed by the definition 
of the set $\mathcal{A}$. Thus, by (A6) we have
$$
\Gamma_0 = \Gamma_0\setminus(Q_1\cup Q_2)\subseteq \partial\Omega\setminus(Q_1\cup Q_2) \,.
$$ 
For the converse inclusion, we start with an arbitrary $\mx\in \partial\Omega\setminus(Q_1\cup Q_2)$.
Since $(\partial\Omega_n)$ converges to $\partial\Omega$ in the sense of Hausdorff
(for compact sets), there exists $\mx_n\in\partial\Omega_n$ such that $(\mx_n)$ converges 
to $\mx$ (cf.~\cite[Subsection 2.2.3.2]{HP18}). Using that the set $Q_1\cup Q_2$ is closed, 
for $n$ large enough we have $\mx_n\in \partial\Omega_n\setminus (Q_1\cup Q_2)=\Gamma_0$. 
Thus, $\mx\in \overline{\Gamma_0}$, which leads to the identity:
$$
\mx\in \bigl(\partial\Omega\setminus(Q_1\cup Q_2)\bigr)\cap\overline{\Gamma_0}
	= \overline{\Gamma_0}\setminus(Q_1\cup Q_2) = \Gamma_0 \,.
$$
The steps are justified by $\overline{\Gamma_0}\subseteq\partial\Omega$
(first equality) and assumption (A6) (second equality).
Therefore, we indeed have $(\Omega,\theta,\mA)\in\mathcal{A}'$.

Before proceeding to the next step, we first show that $\Gamma_1^n=\partial\Omega_n\cap Q_1$
converges, up to  a subsequence,
in the sense of Hausdorff for compact sets to
$\Gamma_1=\partial\Omega\cap Q_1$. 
Let us denote by $L$ the Hausdorff limit of $\Gamma_1^n$. Since the inclusion is preserved under
Hausdorff convergence for compact sets (cf.~\cite[Subsection 2.2.3.2]{HP18}),
we have $L\subseteq\Gamma_1$. In general, the intersection is not continuous with 
respect to Hausdorff convergence for compact sets.
However, the specific structure of our sets allows us to handle this issue.
Let us take $\mx\in\Gamma_1\subseteq\partial\Omega$. Since $(\partial\Omega_n)$
converges in the sense of Hausdorff (for compact sets) to $\partial\Omega$, 
by \cite[(2.14)]{HP18} there exists a sequence $(\mx_n)$ that converges to 
$\mx$ and such that $\mx_n\in\partial\Omega_n$. 
If in addition we have $\mx_n\in Q_1$, then using the same characterisation we get
$\mx\in L$ (of course, it is sufficient to have a subsequence of $(\mx_n)$ that 
it is contained in $Q_1$). Hence, it is left to analyse the case where 
$\mx_n\not\in Q_1$, for all but finitely many $n$.   
Since $Q_1$ and $Q_2$ are disjoint compact sets (see (A6)), for sufficiently 
large $n$ all $\mx_n$ lie on $\overline{\Gamma_0}$. 
This in turn implies $\mx\in\overline{\Gamma_0}$.
Consequently, we have 
$$
\mx\in\partial\Omega\cap  Q_1 \cap\overline{\Gamma_0}
	=   Q_1 \cap\overline{\Gamma_0}
	\subseteq Q_1\cap\partial\Omega_n = \Gamma_1^n \,,
$$
where we have used that for any $n\in\N$ it holds 
$\overline{\Gamma_0}\subseteq \partial\Omega\cap \partial\Omega_n$.
Then we can conclude that $\mx\in L$, and in particular $L=\Gamma_1$.

In the next step we study the boundary value problem (Step II in the proof of Theorem \ref{tm:dir}).
We start as before with homogenizing the Dirichlet boundary conditions. 
And the rational is completely analogous as in the proof of Theorem \ref{tm:neu}.
More precisely, for any $i\in\{1,2,\dots,m\}$ and any $n\in\N$, 
there is a unique solution $v_i^n$ to the following variational problem:
\begin{equation}\label{eq:weak_gen}
\int_{\Omega_n} \mA_n\nabla (v^n_i+H_i)\cdot\nabla\varphi\,d\mx=\int_{\Omega_n} f_i\varphi\,d\mx\,,\;\varphi\in W_n\,,
\end{equation}
where $H_i$ is given in (H3) and
\[
W_n :=\{v\in\Hj{\Omega_n}:v|_{\Gamma_0\cup \Gamma_1^n}=0\} \,.
\]
Recall that $\Gamma_i^n=\partial\Omega_n\cap Q_i$, $i=1,2$.
Then it is easy to check that $u_i^n = H_i|_{\Omega_n} + v_i^n \in H_i|_{\Omega_n} + W_n$
is the unique weak solution to  
\begin{equation}\label{eq:gen_u_n}
\left\{ \begin{array}{cc} -\dv(\mA_n \nabla u_i^n)=f_i& \text{ in }\Omega_n\\ 
u_i^n=h_i &\text{ on }\Gamma_0\\
u_i^n=0 &\text{ on }\Gamma_1^n\\
\mA\nabla u_i^n\cdot \mn=0 &\text{ on }\Gamma_2^n ,
\end{array} \right.
\end{equation}
where we used that $H_i=0$ a.e.~on $Q_1$ and $\Gamma_1^n\subseteq Q_1$.

Analogous to the derivation in \eqref{eq:ocj}, we can show that the sequence of 
real numbers $(\|v_i^n\|_{\Hj{\Omega_n}})_n$ is bounded. Indeed, by
utilising Lemma \ref{lem:poincare} (where $\sqrt{c_1}$ represents the constant $c$), 
the coercivity of $\mA$, the weak formulation \eqref{eq:weak_gen}, and assumption
(H3), we obtain:
\begin{equation*}
\begin{aligned}
\alpha c_1 \|v^n_i\|^2_{\Hj{\Omega_n}} &\leq \alpha \|\nabla v^n_i\|^2_{\Ld{\Omega_n}}\\
&\leq \int_{\Omega_n} \mA_n\nabla v^n_i\cdot\nabla v^n_i \,d\mx
	= \int_{\Omega_n}      f_i v^n_i -\mA_n\nabla H_i\cdot\nabla v^n_i\,d\mx\\
&\leq  \|f_i\|_{\Ld{D}}\| v_i^n\|_{\Ld{\Omega_n}} 
	+\beta \|\nabla H_i\|_{\Ld{D}} \|\nabla v_i^n\|_{\Ld{\Omega_n}}\\
&\leq c_2\|v_i^n\|_{\Hj{\Omega_n}}\,,
\end{aligned}
\end{equation*}
where constant $c_2>0$ does not depend on $n$. Thus, 
$\sup_{n\in\N} \|v_i^n\|_{\Hj{\Omega_n}}< \infty$.

We now need to extend $v_i^n$ to the whole domain $D$. Since $v_i^n$ is 
not necessarily zero on the entire boundary $\partial\Omega_n$ (see $W_n$), 
a simple zero extension (as employed in the proof of Theorem $\text{\ref{tm:dir}}$) 
is not possible. Furthermore, while the extension in the sense of Chenais is applicable, 
it does not preserve the homogeneous Dirichlet boundary conditions required
for our framework. Thus, we will use a combination of these two extension techniques.

Since $Q_1$ and $Q_2$ are compact disjoint sets (see (A6)), there exists open, bounded 
and disjoint open neighbourhoods $U_1, U_2\subseteq\R^d$, i.e.~$Q_i\subset U_i$, $i=1,2$. 
Let us take two smooth nonnegative functions $\varphi_1,\varphi_2\in {\rm C}^\infty_c(\R^d)$ 
with the following properties:
\begin{align*}
&\varphi_i=0 \quad \hbox{on} \ U_j \,, \ i\neq j\,,\\
&\varphi_1^2+ \varphi_2^2 =1  \quad  \hbox{on} \ \R^d
\end{align*}
(the necessity of the condition $\varphi_1^2 + \varphi_2^2 = 1$ will become 
apparent later in the proof). 
This can be achieved by mollifying the characteristic functions of the complement of 
$\widetilde U_i$, $i=1,2$, where $\widetilde U_i$ are slightly larger disjoint open sets 
$U_i\ssubset\widetilde U_i$. If we denote these functions by $\eta_i$, $i=1,2$,
then we can take
$\varphi_1=\frac{\eta_2}{\sqrt{\eta_1^2+\eta_2^2}}$ and 
$\varphi_2=\frac{\eta_1}{\sqrt{\eta_1^2+\eta_2^2}}$ (note that the denominator is 
never zero).

Let us define $w_i^n:=v_i^n\varphi_1$. Since $v_i^n$ is zero on $\Gamma_0\cup\Gamma_1^n$,
while $\varphi_1$ is zero on $U_2\supset Q_2$, we have that $w_i^n$ has the zero trace on
the whole boundary $\partial\Omega_n$. Hence, $w_i^n \in\HH_0^1(\Omega_n)$, which allows
us to extend $w_i^n$ by zero to the whole $D$. Let us denote these extensions by a single 
vector-valued function $\widetilde\vw^n=(\widetilde w_1^n,\dots,\widetilde w_m^n)$. 
Since,
$$
\|\widetilde w_i^n\|_{\Hjn{D}} = \|w_i^n\|_{\Hj{\Omega_n}} 
	\leq C_{\varphi_1} \|v_i^n\|_{\Hj{\Omega_n}} \,,
$$
the sequence $(\widetilde\vw^n)$ is bounded in $\HH_0^1(D;\R^m)$.
Thus, we can pass to another subsequence such that $(\widetilde\vw^n)$
converges to $\widetilde\vw\in\HH_0^1(D;\R^m)$ weakly in $\HH_0^1(D;\R^m)$ and almost everywhere 
on $D$.

For a sequence $z_i^n:= v_i^n\varphi_2$ we apply the extension of Chenais \cite{Chenais}, 
denoted by $\widetilde z_i^n$ and $\widetilde\vz^n=(\widetilde z_1^n,\dots,\widetilde z_m^n)$.
Then by \eqref{eq:chenais} we have
$$
\|\widetilde \vz^n\|_{\Hj{D;\R^m}} \leq \sqrt{c_0} \|\vz^n\|_{\Hj{\Omega_n;\R^m}} 
\leq \sqrt{c_0} C_{\varphi_2} \|\vv^n\|_{\Hj{\Omega_n;\R^m}} \,,
$$
implying that $(\widetilde\vz^n)$ is bounded in $\HH^1(D;\R^m)$.
We then take another subsequence such that $\widetilde\vz^n$ converges to 
$\widetilde\vz$ weakly in $\HH^1(D;\R^m)$ and almost everywhere 
on $D$. 

Finally, let us define
$$
\widetilde\vv^n:= \varphi_1\widetilde\vw^n + \varphi_2\widetilde\vz^n \,.
$$
Then $(\widetilde\vv^n)$ converges to $\varphi_1\widetilde\vw + \varphi_2\widetilde\vz =: \widetilde\vv$
weakly in $\Hj{D;\R^m}$ and a.e.~on $D$. 
Let us show that $\widetilde\vv^n$ is an extension of $\vv^n$. For a.e.~$\mx\in\Omega_n$ 
we have
$$
\begin{aligned}
\widetilde\vv^n(\mx) &= \varphi_1(\mx)\widetilde\vw^n(\mx) + \varphi_2(\mx)\widetilde\vz^n(\mx) \\
&= \varphi_1(\mx)\vw^n(\mx) + \varphi_2(\mx)\vz^n(\mx) \\
&= (\varphi_1(\mx)^2+\varphi_2(\mx)^2) \vv^n(\mx) = \vv^n(\mx) \,.
\end{aligned}
$$
Thus, as announced, we have successfully constructed an extension of
$\vv_n$ using a combination of the zero-extension and Chenais's extension. 
Crucially, the construction method is independent of $n$.

This completes Step II. Now we turn to the restriction 
$\vv:=\widetilde\vv|_\Omega$. 
While it is obvious that $\vv\in\Hj{\Omega;\R^m}$, it is left to check 
the boundary conditions, i.e.~that $\vv$ belongs to the suitable
variational space $W^m$, where
\begin{equation*}
W:=\{v \in\Hj{\Omega}:v|_{\Gamma_0\cup \Gamma_1}=0\}
\end{equation*}
and $\Gamma_1=\partial\Omega\cap Q_1$.

Since the condition $\vv^n|_{\Gamma_0} =\vnul$ holds for any $n\in\N$,
it follows from the construction of the extensions that
$\widetilde\vw^n|_{\Gamma_0} = \widetilde\vz^n|_{\Gamma_0}=\vnul$ for
all $n$ (recall that $\varphi_1, \varphi_2$ are smooth).
As the trace is weakly sequentially continuous, we conclude 
that the weak limits also satisfy the boundary condition:
$\widetilde\vw|_{\Gamma_0} = \widetilde\vz|_{\Gamma_0}=\vnul$.
This, in turn, implies that $\vv$ satisfies the 
homogeneous boundary condition on the fixed part of the boundary: 
$\vv|_{\Gamma_0} = \vnul$.

Following the approach from the beginning of Step
III of the proof of Theorem \ref{tm:dir}, we get that 
$\vw:=\widetilde\vw|_\Omega\in \Hjn{\Omega;\R^m}$.
In particular, $(\varphi_1\vw)|_{\Gamma_1}=\vnul$.
On the other hand, we trivially have 
$(\varphi_2\vz)|_{\Gamma_1}=\vnul$ 
because $\varphi_2$ was constructed to be zero on the 
open set $U_1$, which contains $\Gamma_1$ 
(since $\Gamma_1 \subseteq Q_1 \subseteq U_1$).
The latter is the reason why we did not define 
$\widetilde\vv^n$ simply as the sum of $\widetilde\vw^n$ and 
$\widetilde\vz^n$, as then we would not have any information
about the support of $\widetilde\vz^n$ around $\Gamma_1$. 
Therefore, based on the established boundary conditions, 
we indeed have $\vv\in W^m$. 

Furthermore, at the same subsequence we have weak convergence of $\widetilde\vu^n:=\widetilde\vv^n+\vH$ to $\widetilde\vu$, and in the same manner we denote  $\vu:=\widetilde\vu|_\Omega$, belonging to the desired space 
$\vH+W^m$.

It is left to see that $\vv$ is a solution to a suitable 
variational problem, i.e.~we need to pass to the limit 
in \eqref{eq:weak_gen} (or in \eqref{eq:gen_u_n} if we consider 
the equation for $u_i^n$).

Let us take an arbitrary $\varphi\in\Cbc{\R^d}$
which is equal to zero on an open neighbourhood $U$ of
$\Gamma_0\cup\Gamma_1$, i.e.~$\Gamma_0\cup\Gamma_1\subseteq U$.
The set of restriction to $\Omega$ of such functions $\varphi$
is dense in $W$ (given above). 
Indeed, the claim follows by \cite[Theorem 3.1]{Bern}.
This is because the set $\Gamma_0 \cup \Gamma_1$ is relatively 
open in $\partial\Omega$ (since its complement $\Gamma_2$ is closed 
in $\partial\Omega$). Furthermore, the interface between the sets, 
defined by $\overline{\Gamma_0\cup\Gamma_1} \cap \Gamma_2$, 
simplifies to $\overline{\Gamma_0}\cap\Gamma_2$ (since $Q_1$ and $Q_2$
are disjoint). 
This interface, which is a subset of $\partial\Gamma_0$, 
satisfies the required regularity condition: 
it has finitely many connected components and is locally 
a graph of a Lipschitz continuous function, as ensured by assumption (A3).

Since $\Gamma_1^n$ converges to $\Gamma_1$ in the sense of Hausdorff
(for compact sets), for sufficiently large $n$ we have
$\Gamma_1^n\cup\Gamma_0\subseteq U$ (cf.~\cite[Proposition 2.2.17]{HP18} 
for the complementary property for open sets).
Hence, $\varphi|_{\Omega_n}\in W_n$, i.e.~$\varphi$ 
is an admissible test function for the variational problem for $v_i^n$
\eqref{eq:weak_gen}. 

The passage to the limit in \eqref{eq:weak_gen} for the above chosen
$\varphi$ is performed in exactly the same way as in the proof of 
Theorem \ref{tm:neu}.
Thus, we end up with
\begin{equation*}
\int_{\Omega} \mA\nabla u_i\cdot\nabla\varphi\,d\mx=\int_{\Omega}  f_i\varphi\,d\mx\,,\ i=1,\ldots,m\,,
\end{equation*}
valid for any $\varphi\in{\rm C}^\infty(\overline{D})$ 
that vanish on a neighborhood of $\Gamma_0\cup\Gamma_1$, 
which means that the limit $\vu$ solves \eqref{state1}.

The rest of the proof of Theorem \ref{tm:dir}, i.e.~Step IV, 
carries over without modification.
\end{proof}

\section{Hybrid method - algorithm overview}\label{sec:algorithm}

Building on the relaxation results derived in the previous section, we now introduce the associated numerical scheme. 
The method follows the relaxation framework closely: it combines homogenization-based material relaxation in the 
interior with shape optimization driven by shape differentiation of the domain.

To ensure well-posedness of the algorithm, we impose additional regularity assumptions on the cost functional, namely 
its differentiability. Specifically, we assume that the functions  $g_\alpha$ and $g_\beta$ are differentiable on $D\times\R^m$ 
with respect to both the spatial variable $\mx$ and  the state variable $\vu$.
We denote the gradient of $g_\gamma$ (for $\gamma=\alpha,\beta$) with respect to  $\mx$ and $\vu$ by  
$\nabla_\mx g_\gamma=(\frac{\partial g_\gamma}{\partial x_1},\hdots,\frac{\partial g_\gamma}{\partial x_d})^\tau$, 
and $\nabla_\vu g_\gamma=(\frac{\partial g_\gamma}{\partial u_1},\hdots,\frac{\partial g_\gamma}{\partial u_m})^\tau$, respectively.

We assume the following:

\begin{itemize}
	\item[(A7)] % $g_{\gamma},\frac{\partial g_{\gamma}}{\partial x_i},\frac{\partial g_{\gamma}}{\partial u_j}$ are Carathéodory functions on $D\times\R^m$ for ${\gamma=\alpha,\beta,\: i=1,\hdots,d,\: j=1,\hdots,m}$ satisfying the following  growth conditions:
\begin{itemize}
\item[$i.$] For $\gamma\in\{\alpha,\beta\}$ and $i=1,2,\hdots,n$, $\frac{\partial g_{\gamma}}{\partial x_i}:D\times\R^m\to\R$  are Carathéodory functions and satisfy the same growth condition as $g_{\gamma}$ in (A5).
\item[$ii.$] For $\gamma\in\{\alpha,\beta\}$ and $j=1,2,\hdots,m$, $\frac{\partial g_{\gamma}}{\partial u_j}:D\times\R^m\to\R$  are Carathéodory functions and  satisfy the  growth conditions
\begin{equation*} %\label{assump2}
	\left|\frac{\partial g_{\gamma}}{\partial u_j}(\mx,\vu)\right|\leq\widetilde \varphi_{\gamma}(\mx)+\widetilde\psi_{\gamma}(\mx)|\vu |^{q-1},   %\quad \gamma=\alpha,\beta, 
\end{equation*}
for some $q\in[1,q^\ast\rangle$, with nonnegative functions $\widetilde\varphi_{\gamma}\in\pL {s_1}D$ 
   and $\widetilde\psi_{\gamma}\in\pL {s_2}D$, where $s_1>q_*'$ (with $q_*'$ denoting the dual exponent), and  $s_2$ satisfies the same condition as in (A5).
\end{itemize}

\end{itemize}

%We provide here only a brief presentation; for a detailed exposition, the reader is referred to \cite{All02}. 

We begin by fixing a domain $\Omega\in\lO'$ (or $\lO$ if some of the variable boundary conditions are absent, 
as discussed in  Remark \ref{rem:Q1Q2}). For a fixed $\Omega$, we employ the optimality criteria method, 
a well-established approach in structural optimization \cite{R89, B95, All02}. Within the framework of compliance 
minimization in linearized elasticity \cite{AF93}, this method can be rigorously interpreted as a descent algorithm 
of the alternating-directions type for a double minimization problem.  The convergence of the method for energy 
minimization in stationary diffusion problems involving two isotropic materials was proved in \cite{T97}, and for 
the multi-state case in \cite{BC20}.  In the specific case of mixtures of two isotropic phases, the explicit solution 
of the G-closure problem enables the treatment of a  general cost functional with multiple states 
\cite{All02,Vnarwa,BCV18}.

The objective is to determine an optimal distribution of materials such that {($\Omega$ is fixed)} the 
admissible design set $\lA_\Omega:=\{ (\theta,\mA) :  (\Omega,\theta,\mA)\in{\mathcal{A'}}\}$ minimizes the prescribed 
objective functional.

Let $(\theta,\mA)\in \lA_\Omega$, and let $\vu$ denote the solution of \eqref{state1} associated with the 
triplet $(\Omega,\theta,\mA)\in \mathcal{A'}$. Consider a variation $(\delta\theta,\delta\mA)$ of the given design. 
The corresponding first-order variation (more precisely, its G\^{a}teaux derivative)  of the cost functional $\JJ$ 
with respect to $(\theta,\mA)$  is then given by (for details, the reader is referred to \cite{All02})
\[ \JJ'_{(\theta,\mA)}(\Omega,\theta,\mA;\delta\theta,\delta\mA) 
=\int_\Omega \delta\theta [g_\alpha(\cdot,\vu)-g_\beta(\cdot,\vu)] \,\mathrm{d}\mx
-\int_\Omega \sum_{i=1}^m \delta A \nabla u_i\cdot \nabla p_i\,\mathrm{d}\mx. \]
Here, the adjoint state function $\vp=(p_1,\ldots,p_m)\in W^m=\{ \vu\in \HH^1(\Omega;\R^m)  : \vu|_{\Gamma_0\cup\Gamma_1}
=\vnul \}$  are solutions of the following adjoint equation: 
\begin{equation}\label{eq:adjoint}
   (\forall \varphi \in V)\quad \int_\Omega \mA \nabla p_i\cdot\nabla \varphi\,\mathrm{d}\mx  
   =\int_\Omega \left[\theta \frac{\partial g_\alpha}{\partial u_{i}}(\cdot,\vu)
   +(1-\theta)\frac{\partial g_\beta}{\partial u_i}(\cdot,\vu)\right]\varphi\,\mathrm{d}\mx,
\end{equation}
for $i=1,\hdots,m$, with $W=\{\varphi\in\Hj{\Omega}:\varphi|_{\Gamma_0\cup \Gamma_1}=0\}$. 
%Note that $\vu$ and $\vp$ will always represent the solution of \eqref{state1} and \eqref{eq:adjoint}, respectively, for a given triplet $(\Omega,\theta,\mA)\in {\mathcal{A}_\varepsilon}$, unless stated otherwise.
%For simplicity, w

\begin{remark}\label{remarkGrowth}
The right-hand side of \eqref{eq:adjoint} is well defined, that is,
\[  \theta \frac{\partial g_\alpha}{\partial u_{i}}(\cdot,\vu)
   +(1-\theta)\frac{\partial g_\beta}{\partial u_i}(\cdot,\vu)\]
belongs to $\pL{r}{\Omega}$, for some $r>q_*'$, which is continuously embedded in $W'$, the dual space of $W$.
%Indeed, since $\theta\in L^\infty(\Omega;[0,1]),$ it is sufficient to prove that $\frac{\partial g_\gamma}{\partial u_{i}}(\cdot,\vu)$ belongs to $L^{r}(\Omega),$ for some $r>q_*'$. 
Indeed, %by the growth condition (A7), we have
%\[\left|\frac{\partial g_{\gamma}}{\partial u_j}(\mx,\vu(\mx))\right|\leq\widetilde \varphi_{\gamma}(\mx)+\widetilde\psi_{\gamma}(\mx)|\vu(\mx) |^{q-1}\quad \text{a.e.}\:\: \mx\in \Omega\]
%where 
we have $\vu\in W^m \hookrightarrow  \pL{q_\ast} {\Omega;\R^m} $ 
(or $\vu\in\pL{p} {\Omega;\R^m}$ for any $p\in \left[1,+\infty\right>$ if $d=2$). 

In the case $d=2$, we have
$|\vu|^{q-1}\in\pL{\frac p{q-1}}{\Omega}$ for any $p\in \left[1,+\infty\right>$. Consequently,
$\widetilde\psi_{\gamma}|\vu|^{q-1}\in\pL{r}{\Omega}$ whenever 
$\frac{1}{r}=\frac{1}{s_2}+\frac{q-1}{p}$.
Since $p$ can be chosen arbitrarily large, the right-hand side converges to $\frac{1}{s_2}$ as $p\to+\infty$. In particular, we can ensure  $r>q_\ast'=1$.

For  $d>2$, Hölder's inequality yields $\widetilde\psi_{\gamma}|\vu|^{q-1}\in\pL{r}{\Omega}$  provided that %, for some $r>q_*'$
$ \frac{1}{r}=\frac{1}{s_2}+\frac{q-1}{q_\ast}$.
Under assumption (A7), this exponent satisfies
\[ \frac{1}{r}=\frac{1}{s_2}+\frac{q-1}{q_\ast}<\frac{q_\ast-1}{q_\ast}=\frac{1}{q_*'} ,\]
and hence $r>q_*'$.

%, and the same conclusion holds.
%Hence, $p$ must satisfy
%\[ p>\frac{q-1}{1-\frac{1}{q_\ast}-\frac{1}{s_2}} \]
%Since we also have $p<q_\ast$, it remains to verify that
%\[\frac{q-1}{1-\frac{1}{q_\ast}-\frac{1}{s_2}}< q_\ast \]
%is well defined. This condition is equivalent to the assumption on $s_2$ in (A5). Therefore, there exists $p>1$ such that the expression on the
%right-hand side belongs to $\pL{r}{\Omega},$ for some $r>q_*'$. 
%The growth conditions on the partial derivatives with respect to $\mx$ are required only for Theorem~\ref{thm:hybridShape}, 
%and not for the homogenization method. In particular, they ensure that 
%$\frac{\partial g_\gamma}{\partial u_{i}}(\cdot,\vu)$ belongs to $\pL{1}{\Omega}$  (see Step~IV in the proof of Theorem~\ref{tm:dir}).
\end{remark}

We define the matrix function 
\begin{equation}\label{eq:M}
\mM:= \text{Sym} \sum_{i=1}^m \nabla u_i \otimes \nabla p_i
\end{equation}
where $\text{Sym}$ denotes the symmetric part of the matrix. We also define the mapping:
\begin{equation}\label{eq:mapF}
     F(\theta,\mM) := \max_{\mA \in \mathcal{K}(\theta)} \mA:\mM. 
\end{equation}

The constraint on the quantities of the original phases is incorporated in the standard way through the associated Lagrangian functional:
\begin{equation}
   \lL(\Omega,\theta,\mA)={\JJ}(\Omega,\theta,\mA) +l\int_{\Omega}\theta\,\mathrm{d}\mx\\
\end{equation}
where $l$ is the Lagrangian multiplier. 
The following necessary conditions of optimality are satisfied  (see \cite[Theorem 3.2.14]{All02}):
\begin{theorem}
    For an optimal $(\theta^\ast,\mA^\ast)\in \lA_\Omega$ and the respective solutions $\vu^\ast$ and $\vp^\ast$ of \eqref{state1} and \eqref{eq:adjoint} corresponding to the triplet $(\Omega,\theta^\ast,\mA^\ast)$, define
    \[ q(\mx) := l+g_\alpha(\mx,\vu^\ast(\mx))-g_\beta(\mx,\vu^\ast(\mx))-\frac{\partial F}{\partial \theta} (\theta^\ast(\mx),\mM^\ast(\mx)), \]
    with $\mM^\ast$ defined by \eqref{eq:M} with $u_i^*$ and $p_i^*$ on the right-hand side.
    The optimal density $\theta^\ast$ satisfies the following conditions, almost everywhere on $\Omega$ 
    \[\begin{array}{cl}
         q(\mx)>0 \quad&\implies\quad \theta^\ast(\mx)=0,  \\
         q(\mx)<0 \quad&\implies\quad \theta^\ast(\mx)=1,  \\
         q(\mx)=0 \quad&\implies\quad \theta^\ast(\mx)\in[0,1].  \\
    \end{array}\]
    Furthermore, $\mA^\ast$ is a maximizer in the definition of $F(\theta^\ast,\mM^\ast)$.
\end{theorem}

%We now focus on the case $d=2$. For readability, we restate a result from \cite{All02} Lemma 3.2.17 
%\begin{prop}
%    Let $\mu_1,\mu_2$ be eigenvalues of $M$ such that $\mu_1\leq \mu_2$. Then function defined by \eqref{eq:mapF} is equal to 
%   \[ f(\theta,M)=\left\{\begin{array}{cc}
%        \beta(\mu_1+\mu_2)-\frac{\beta(\beta-\alpha)\theta}{\alpha+\beta +(\beta-\alpha)\theta}(\sqrt{\mu_1}+\sqrt{\mu_2})^2 & \text{ if }\: \:  0<\mu_1,\:\:\theta<\theta_B  \\
%        \alpha(\mu_1+\mu_2)-\frac{\alpha(\beta-\alpha)(1-\theta)}{2\alpha+(\beta-\alpha)\theta}(\sqrt{-\mu_1}+\sqrt{-\mu_2})^2  &  \text{ if }\: \:  \mu_2<0,\:\:\theta>\theta_A  \\
%        \mu_1\frac{\alpha\beta}{\theta(\beta-\alpha)+\alpha}+\mu_2(\beta+\theta(\beta-\alpha)) & \text{ otherwise},
%   \end{array}\right. \]
%   where $\theta_A=\frac{\alpha}{\beta-\alpha}\left( \frac{\sqrt{-\mu_1}}{\sqrt{-\mu_2}}-1\right)$ and $\theta_B=\frac{\beta}{\beta-\alpha}\left( \frac{\sqrt{\mu_1}}{\sqrt{\mu_2}}-\frac\alpha\beta\right)$.
%\end{prop}

The function $F$ admits an explicit representation in both two- and three-dimensional cases \cite{All02, Vnarwa}. For any symmetric matrix $\mM$ the mapping $\theta\mapsto\frac{\partial F}{\partial \theta}(\theta,\mM)$ is a monotone and piecewise smooth. Moreover, for fixed matrix $\mM$ and a real number $c$, the equation $\frac{\partial F}{\partial \theta}(\theta,\mM)=c$ with respect to  $\theta$ 
reduces to a quadratic equation and can therefore be solved explicitly.

The second component of the hybrid approach consists in modifying the boundary of the domain. We adopt an approach based on the shape derivative of the objective functional. This is formulated via a perturbation of the identity, $\MAP = \ID + \psi$, where $\psi \in W^{1,\infty}(\R^d,\R^d)$ is sufficiently small in norm to ensure that $\MAP$ is a homeomorphism. A shape functional $\lL$, defined on a family of open subsets compactly contained in $D$ and satisfying the $\varepsilon$-cone property, is said to be shape differentiable if the mapping
\[\psi\mapsto\mathcal{L}(\MAP(\Omega)) \] 
is Fréchet differentiable at zero. The Fréchet derivative at zero in the direction $\psi$ is denoted by $\mathcal{L}'(\Omega;\psi)$ and is referred to as the shape derivative of $\mathcal{L}$ at $\Omega$ in the direction $\psi$.

Given a triplet $(\Omega,\theta,\mA)\in \mathcal{A'}$, we consider its perturbation
$(\MAP(\Omega_n),\theta\circ\MAP^{-1},\mA\circ\MAP^{-1})$ 
and denote by $\vu(\psi) := \big(u_1(\psi), \ldots, u_m(\psi)\big) \in 
\HH^1(\MAP(\Omega);\R^m)$ the associated perturbed state, that is, the vector of temperature 
fields. 
{ We additionally assume that $\psi=0$ on $\Gamma_0$ in order to keep this  portion of the boundary fixed.}
For $i=1,\ldots,m$, the function $u_i(\psi)$ is defined as the solutions of the following 
boundary value problem 
\begin{equation}\label{PerState}%\tag{PerS}
 \left\{ \begin{array}{rl} -\dv(\mA\circ \MAP^{-1} \nabla u_i(\psi))=f_i& \text{ in }\MAP(\Omega)\\ 
  u_i(\psi)=h_i &\text{ on }\MAP(\Gamma_0)=\Gamma_0\\
  u_i(\psi)=0 &\text{ on }\MAP(\Gamma_1)\\
 \mA\circ \MAP^{-1}\nabla u_i(\psi)\cdot \mn=0 &\text{ on }\MAP(\Gamma_2).
 \end{array} \right.
\end{equation}
%For notational convenience, we set $\vu(\psi) := \vu(\cdot;\psi)$ and $u_i(\psi) := u_i(\cdot;\psi)$. 
In particular, for $\psi=0$ the corresponding state $\vu = \vu(0)$ coincides with the solution of 
\eqref{state1}, hence,  we omit the explicit dependence on zero. For the results below, we assume
$f_i\in \HH^1(D)$ and $h_i$ is the trace to $\Gamma_0$ of an $\HH^1$ function vanishing on $\Gamma_1$.
We further introduce the following shape functional:
\begin{equation}\label{eq:shapeFunctional}\begin{split} \mathcal{L}(\MAP(\Omega))
    &= {\JJ}(\MAP(\Omega_n),\theta\circ\MAP^{-1},\mA\circ\MAP^{-1}) 
    +l\int_{\MAP(\Omega)}\theta\circ\MAP^{-1} \,\mathrm{d}\mx\\
&=\int_{\MAP(\Omega)} \theta\circ\MAP^{-1} g_\alpha (\cdot,\vu(\psi))
+(1-\theta\circ\MAP^{-1}) g_\beta (\cdot,\vu(\psi))\,\mathrm{d}\mx\\
&\phantom{=}+l\int_{\MAP(\Omega)}\theta\circ\MAP^{-1} \,\mathrm{d}\mx.
\end{split}\end{equation}

\begin{prop}\label{prop:material}
Let $(\Omega,\theta,\mA)\in{\mathcal{A}}'$. For  $\psi\in \WW^{1,\infty}(\R^d;\R^d)$ such that $\psi=0$ on $\Gamma_0$, the mapping 
\[ \psi \mapsto \vu(\psi)\circ\MAP: \WW^{1,\infty}(\R^d;\R^d)\to \HH^{1}(\Omega;\R^m)\] 
is Fréchet differentiable at zero. Its Fréchet derivative at zero in direction $\psi$ is denoted by $\dot \vu (\psi)=(\dot u_1(\psi),\ldots, \dot u_m(\psi))\in W^m$.
%=\{ \vu\in \HH^{1}(\Omega;\R^m) : \vu|_{\Gamma_0}=0, \vu|_{\Gamma_1}=0 \}$. 
For each $i=1,\ldots,m$, the variational formulation satisfied by $\dot u_i(\psi)$ is given by
\begin{align*}
    \int_\Omega \mA \nabla \dot u_i(\psi) \cdot \nabla \varphi\, \mathrm{d}\mx 
    &=   \int_\Omega (\mA\nabla\psi^\tau+\nabla\psi\mA-\ddiv(\psi)\mA)\nabla u_i \cdot \nabla \varphi\,\mathrm{d}\mx\\
    &\phantom{=}+\int_\Omega \varphi\nabla f_i \cdot\psi+f_i\varphi\ddiv\psi\,\mathrm{d}\mx
\end{align*}
for any $\varphi  \in V$, where $u=(u_1,\ldots,u_m)$ denotes the solution of \eqref{state1}.
\begin{proof}
The result follows from Theorem~3.1 in \cite{KunVrd22}. The fact that $\mA$ is no longer piecewise constant scalar matrix does not alter the overall argument, except that $\mA$ and $\nabla \psi$ no longer commute, which results in a modified expression. Furthermore, due to the prescribed boundary conditions, the implicit function theorem is applied in the space 
$W \subset H^1(\Omega)$ rather than in $H_0^1(\Omega)$. 
\end{proof}
\end{prop}

\begin{theorem}\label{thm:hybridShape}
Let  $\psi \in W^{1,\infty}(\R^d;\R^d)$ be such that $\psi=0$ on $\Gamma_0$. Assume that $g_\alpha$ and $g_\beta$ satisfy {\rm (A5)} and {\rm (A7)}. Then the shape functional defined in \eqref{eq:shapeFunctional} is shape differentiable. Moreover, for a given triplet $(\Omega,\theta,\mA) \in \mathcal{A}'$, its shape derivative is given by
%\begin{equation}
%    \begin{split}
%        \mathcal{L}'(\Omega;\psi)&= \sum_{i=1}^m \int_\Omega  (\mA\nabla\psi^\tau+\nabla\psi\mA-\ddiv(\psi)\mA)\nabla u_i\cdot \nabla p_i +\ddiv(fp_i)u_i \,\mathrm{d}\mx\\
%        &\phantom{=}+\int_\Omega \theta \nabla_\mx g_\alpha(\cdot,\vu(\cdot))\cdot\psi + (1-\theta)\nabla_\mx g_\beta(\cdot,\vu(\cdot))\cdot\psi\,\mathrm{d}\mx\\
%        &\phantom{=}+\int_\Omega [\theta g_\alpha(\cdot,\vu(\cdot))+(1-\theta)g_\beta(\cdot,\vu(\cdot))+l\theta]\ddiv \psi\,\mathrm{d}\mx
%    \end{split}
%\end{equation}
\begin{equation}
    \begin{split}
        \mathcal{L}'(\Omega;\psi)&= \int_\Omega  (\mA\nabla\psi^\tau+\nabla\psi\mA-\ddiv(\psi)\mA):\mM  \,\mathrm{d}\mx\\%+\vp\cdot(\nabla\vf\psi)+\vp\cdot\vf\ddiv(\psi) \,\mathrm{d}\mx\\
        &\phantom{=}+\int_\Omega \theta \nabla_\mx g_\alpha(\cdot,\vu)\cdot\psi + (1-\theta)\nabla_\mx g_\beta(\cdot,\vu)\cdot\psi\,\mathrm{d}\mx\\
        &\phantom{=}+\int_\Omega [\theta g_\alpha(\cdot,\vu)+(1-\theta)g_\beta(\cdot,\vu)+l\theta]\ddiv \psi\,\mathrm{d}\mx \\
        &\phantom{=}+\sum_{i=1}^m \int_\Omega p_i\nabla f_i \cdot\psi +f_ip_i\ddiv\psi\,\mathrm{d}\mx,
    \end{split}
\end{equation}
where $\vu=(u_1,\hdots,u_m)$ is the solution of \eqref{state1}, $\mM$  is given by \eqref{eq:M}, and $\vp=(p_1,\hdots,p_m)\in W^m$ is  the  solution of \eqref{eq:adjoint}.
\end{theorem}

\begin{proof}
The shape functional defined in \eqref{eq:shapeFunctional} should be transformed to the original domain $\Omega$, i.e.
\begin{align*}
    \mathcal{L}(\MAP(\Omega))&=\int_{\MAP(\Omega)} \theta\circ\MAP^{-1} g_\alpha (\cdot,\vu(\psi))+(1-\theta\circ\MAP^{-1})g_\beta (\cdot,\vu(\psi))\,\mathrm{d}\mx\\
&\phantom{=}+l\int_{\MAP(\Omega)}\theta\circ\MAP^{-1} (\mx)\,\mathrm{d}\mx.\\
&=\int_\Omega \theta(\mx) g_\alpha (\MAP(\mx),\vu(\psi)\circ \MAP(\mx)) \det \nabla \MAP(\mx)\,\mathrm{d}\mx\\
&\phantom{=}+\int_\Omega(1-\theta(\mx)) g_\beta (\MAP(\mx),\vu(\psi)\circ \MAP(\mx))\det \nabla \MAP(\mx)\,\mathrm{d}\mx\\
&\phantom{=}+l\int_{\Omega} \theta(\mx)  \det \nabla \MAP(\mx)\,\mathrm{d}\mx \,.\\
\end{align*}
Owing to assumption (A7),  we have $\nabla_\mx g_\gamma(\cdot,\vu)\in\pL{1}{\Omega;\R^d}$ (by the same arguments as in  Step~IV of the proof of Theorem~\ref{tm:dir}). 
Moreover, for some $r>q_*'$ (see Remark~\ref{remarkGrowth}), we have $\nabla_\vu g_i(\cdot,\vu)\in\pL{r}{\Omega;\R^m}$. Since
 $\dot \vu(\psi) \in {\rm H}^1(\Omega;\R^m)$, it follows that the mapping $\psi \mapsto g_i(\cdot,\vu(\psi))\circ \MAP$ is Fréchet differentiable at zero for $\gamma=\alpha,\beta$. More precisely,
\begin{align*}
g_i(\cdot,\vu(\psi))\circ \MAP-g_i(\cdot,\vu)=\nabla_\mx g_i(\cdot,\vu)\cdot\psi + \nabla_\vu g_i(\cdot,\vu)\cdot \dot \vu(\psi)+o(\psi) \quad \text{ in }{\rm L}^1(\Omega) \,.
\end{align*}
From this it follows that
\begin{align*}
    \mathcal{L}(\MAP(\Omega))-\mathcal{L}(\Omega)&=\mathcal{L}'(\Omega;\psi) +o(\psi) \,,
\end{align*}
where
\begin{align*}
   \mathcal{L}'(\Omega;\psi)&=\int_\Omega \left[\theta \nabla_\vu g_\alpha(\cdot,\vu)+(1-\theta) \nabla_\vu g_\beta(\cdot,\vu)\right]\cdot\dot \vu(\psi)\,\mathrm{d}\mx\\
    &\phantom{=}+\int_\Omega \left[\theta\nabla_\mx g_\alpha(\cdot,\vu)+(1-\theta)\nabla_\mx g_\beta(\cdot,\vu)\right]\cdot\psi \,\mathrm{d}\mx\\
    &\phantom{=}+\int_\Omega \bigl(\theta g_\alpha(\cdot,\vu)+(1-\theta)g_\beta(\cdot,\vu)+l\theta\bigr)\ddiv \psi \,\mathrm{d}\mx+o(\psi) \,.
\end{align*}
Note that the first term in $\mathcal{L}'(\Omega;\psi)$ can be rewritten as
\begin{align*}
   \sum_{i=1}^m\int_\Omega \left[\theta \frac{\partial g_\alpha}{\partial u_{i}}(\cdot,\vu)+(1-\theta)\frac{\partial g_\beta}{\partial u_i}(\cdot,\vu)\right]\dot u_i(\psi)\,\mathrm{d}\mx
  &= \sum_{i=1}^m\int_\Omega \mA \nabla p_i \cdot \nabla\dot u_i(\psi)\,\mathrm{d}\mx\\ 
  &\hspace{-3cm}=\sum_{i=1}^m \int_\Omega\mA \nabla \dot u_i(\psi) \cdot \nabla p_i\,\mathrm{d}\mx\\
  &\hspace{-3cm}=\sum_{i=1}^m\int_\Omega (\mA \nabla\psi^\tau +\nabla\psi\mA -\ddiv\psi\mA)\nabla u_i\cdot \nabla p_i\,\mathrm{d}\mx\\
  &\hspace{-3cm}\phantom{=}+\sum_{i=1}^m \int_\Omega p_i\nabla f_i \cdot\psi+f_ip_i\ddiv\psi\,\mathrm{d}\mx\\
  &\hspace{-3cm}=\int_{\Omega} (\mA \nabla\psi^\tau +\nabla\psi\mA -\ddiv\psi\mA) : \mM\,\mathrm{d}\mx \\
  &\hspace{-3cm}\phantom{=}+\sum_{i=1}^m \int_\Omega p_i\nabla f_i \cdot\psi +f_ip_i\ddiv\psi\,\mathrm{d}\mx,
\end{align*}
thus completing the proof.
\end{proof}
%\qed
%\begin{remark}
%Note that the previous results can be extended to the more general case where  $\Omega\ssubset D$ is a Lipschitz domain and the boundary portions $\Gamma_i$ may satisfy $\overline\Gamma_i\cap\overline\Gamma_j\neq \emptyset$.  It is sufficient that each $\Gamma_i$ is relatively open in  $\partial\Omega$, has positive $(d-1)$-dimensional Hausdorff measure,  and consists of finitely many connected components.
% The key observation is that a perturbation generated by  $\psi$, provided it is sufficiently small in the $W^{1,\infty}(\mathbb{R}^d;\mathbb{R}^d)$-norm, does not change the topology of the boundary. 
% In particular, for $\psi$ in a neighborhood of zero, the mapping
%\[
%\psi \longmapsto u(\psi)\circ \Phi_\psi
%\]
%is well defined. This holds not only with respect to geometric 
%variations of $\partial\Omega$, but also with respect to the induced 
%positions of the boundary conditions. 
%Theorem~\ref{thm:hybridShape} holds in the general case, without the assumption that $\Gamma_0$ is fixed. It suffices to extend the assumptions to a larger set $D$ in which $\Omega$ is compactly embedded.
%\end{remark}

The algorithm can be formulated by combining the optimality criteria method for the interior design with the shape derivative of the functional with respect to the overall domain $\Omega$. The set $\Omega$ is represented by a level set function, which enables a straightforward control of the domain volume through suitable modifications of the level set constant.

\begin{algorithmm}\label{alg} \rm
Take initial domain $\Omega^0$. Repeat for $k=0,\hdots,k_S:$
\begin{itemize}
    \item Initalize $\theta^{k,0}=1, \mA^{k,0}=\alpha\I$. Repeat for $i=1,\hdots,k_H$:
    \begin{enumerate}
        %\item 
        \item Calculate solution $\vu^{k,i}$ of \eqref{state1} and $\vp^{k,i}$ of \eqref{eq:adjoint} for the triplet $(\Omega^k,\theta^{k,i-1},\mA^{k,i-1})\in{\mathcal{A}}$
        \item Define $\mM^{k,i}=\text{Sym} \sum_{j=1}^m \nabla u^{k,i}_j \otimes \nabla p^{k,i}_j$
        \item For each $\mx\in\Omega^k$ let $\theta^{k,i}(\mx)$ be  the zero of map 
        \[
        \theta\mapsto l+g_\alpha(\mx,\vu^{k,i}(\mx))-g_\beta(\mx,\vu^{k,i}(\mx))
        -\frac{\partial F}{\partial \theta} (\theta,\mM^{k,i}(\mx)),
        \]
        in segment $[0,1]$. 
        If such a zero does not exist, set
        $\theta^{k,i}(\mx)=0$ if the function is  positive or  $\theta^{k,i}(\mx)=1$  if it is negative.  The Lagrange multiplier $l$ is updated within the above mapping in order to enforce the volume constraint
         \[
         \int_{\Omega^k} \theta^{k,i}(\mx)\,\mathrm{d}\mx = q_\alpha .
         \]
        \item For $\mx\in\Omega^k$ we define $\mA^{k,i}(\mx)\in \mathcal{K}(\theta^{k,i}(\mx))$ as a  maximizer in \eqref{eq:mapF}.
    \end{enumerate}
    \item Set $\theta^k=\theta^{k,k_H}, \mA^{k}=\mA^{k,k_H}.$ Calculate solution $\vu^{k}$ of \eqref{state1} and $\vp^{k}$ of \eqref{eq:adjoint} for the triplet $(\Omega^k,\theta^{k},\mA^{k})\in{\mathcal{A'}}$ and $\mM^{k}=\text{Sym} \sum_{j=1}^m \nabla u^{k}_j \otimes \nabla p^{k}_j$.
    \item %Calculate direction $\psi^k\in\HH_0^1(D,\R^d)$ satisfying
 %   \[(\forall \varphi\in\HH_0^1(D,\R^d) )\quad \int_D \nabla\psi^k\cdot\nabla\varphi+\psi^k\varphi\,\mathrm{d}\mx =-\mathcal{L}'(\Omega^k;\varphi). \]
 Calculate the direction $\psi^k$ as the solution of the boundary value problem
 	\[\left\{\begin{array}{ll} -\Delta\psi^k +\psi^k =0&\text{ in }D \\ \psi^k=0 &\text{ on }\partial D. \end{array}\right.\]
   \item Update $\Omega^{k+1}=(\ID +\psi)(\Omega^k).$
\end{itemize}
\end{algorithmm}

\section{Numerical results}\label{sec:num_ex}
Let $D=[-2.6,2.6]^2\setminus B((0,0);1)$ represent a square domain excluding a unit circle centered at the origin. The isotropic conductivities are $\alpha=1$ and $\beta=2$. The initial design $(\Omega_0,\chi_0)$ consists of $\Omega_0$, a shifted square defined as $\Omega_0=[-1.752,1.792]^2\setminus B((0,0);1)$, which represents a square region excluding the unit circle at its center, and $\chi_0=\alpha \mI$.   The boundary $\Gamma_0=\partial B((0,0);1)$ is fixed, while the outer boundary $\Gamma_1$ is allowed to change. Homogeneous Dirichlet boundary conditions are assumed on both boundaries.  The measure of the domain is $|\Omega|=3\pi$ and the amount of the first phase is given by $q_\alpha=\int_\Omega \chi \, \mathrm{d}\mx=\frac{3\pi}{2}.$ 
We consider two multiple state equations \eqref{state1} where right-hand sides are radial functions
\begin{align*}
     f_1(\mr)=\left\{\begin{array}{cc}
             1, & \mr<2 \\
             1-\mr,  &  \mr\geq2
        \end{array} \right.,
    \quad  f_2(\mr)=\left\{\begin{array}{cc}
             \frac{8(2-\mr)^2}{\mr}, & \mr<2 \\
             0,  &  \mr\geq2
        \end{array} \right. 
\end{align*}
and objective functional 
\[ {\JJ}(\Omega,\chi)= -\int_\Omega f_1(\mx) u_1(\mx) +f_2(\mx)u_2(\mx)\, \mathrm{d}\mx. \]

It is known that for fixed annulus $\Omega=B((0,0);2)\setminus \overline{B((0,0);1)}$, the optimal design is radial and classical (see \cite{KunVrd20}). Specifically, the less conductive material is placed along the inner and outer boundaries, while the more conductive material occupies the intermediate region. The interfaces between the materials are circular, with radii
$1.1928$ and $1.7096$.

The results of the hybrid optimization algorithm are shown in Figure \ref{fig:hybrid_results}.
 The numerical simulations were performed on a mesh of approximately 3,500 triangles, with the final result obtained after refinement on a mesh of 53,000 triangles. In all simulations, the measure $|\Omega_k|$ was kept constant using the \texttt{mmg2d\_O3} library \cite{MMG} and the \texttt{mshdist} algorithm \cite{DapFre12}. 

We observe that the sequence of domains $\Omega_k$ converges to the annulus $\Omega$ described  above, with
the final distribution of the two materials matching the theoretical predictions: the vertical red lines are located at $x$-coordinates $1.1928$ and $1.7096$.
%, indicating that the results obtained align with those predicted by theoretical methods. 
The hybrid method was implemented using a single iteration of the homogenization process ($k_H=1$ in Algorithm \ref{alg}), and it is generally recommended to keep this number low (fewer than five) to accelerate computations. Only in the final iteration ($k=k_S=400$ in Algorithm \ref{alg}), once the optimal $\Omega$ has been identified, it is refined using a larger number of homogenization iterations (30 in our case). The corresponding values of the objective functional are presented in Figure \ref{fig:J}.

\begin{figure}[h]
    \centering
    \includegraphics[width=1\textwidth]{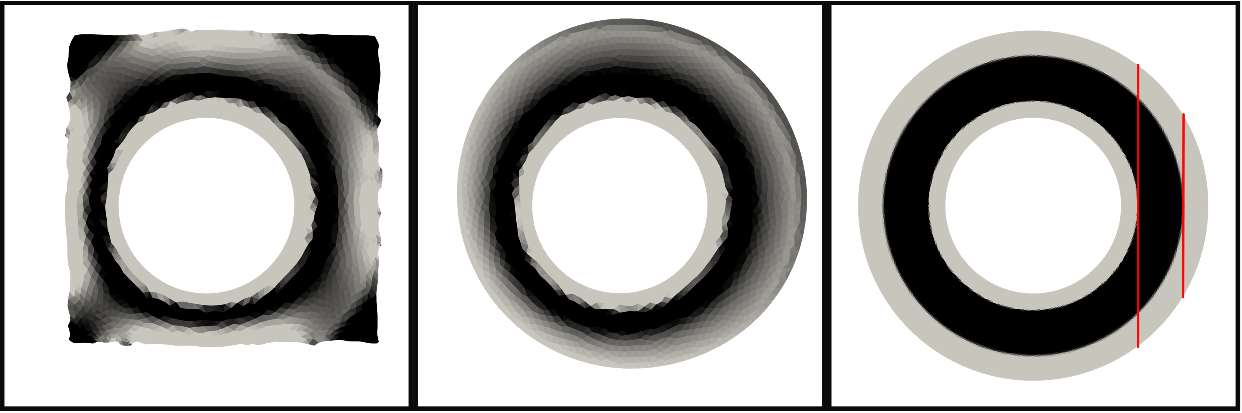}
%    \caption{Results from the hybrid algorithm, local fraction $\theta$ in the $k$-th iteration: $k=1$ left, $k=200$ the middle, 
%$k=400$
%    right. The better conductor is presented by black color, and the worse by grey}
    \caption{Results of the hybrid algorithm: $\Omega$ and local fraction $\theta$ at the $k$-th iteration ($k=1$, left; $k=200$, middle; $k=400$, right). The higher-conductivity phase is shown in black, and the lower-conductivity phase in grey.}
    \label{fig:hybrid_results}
\end{figure}

\begin{figure}[h]
    \centering
    \includegraphics[width=1\textwidth]{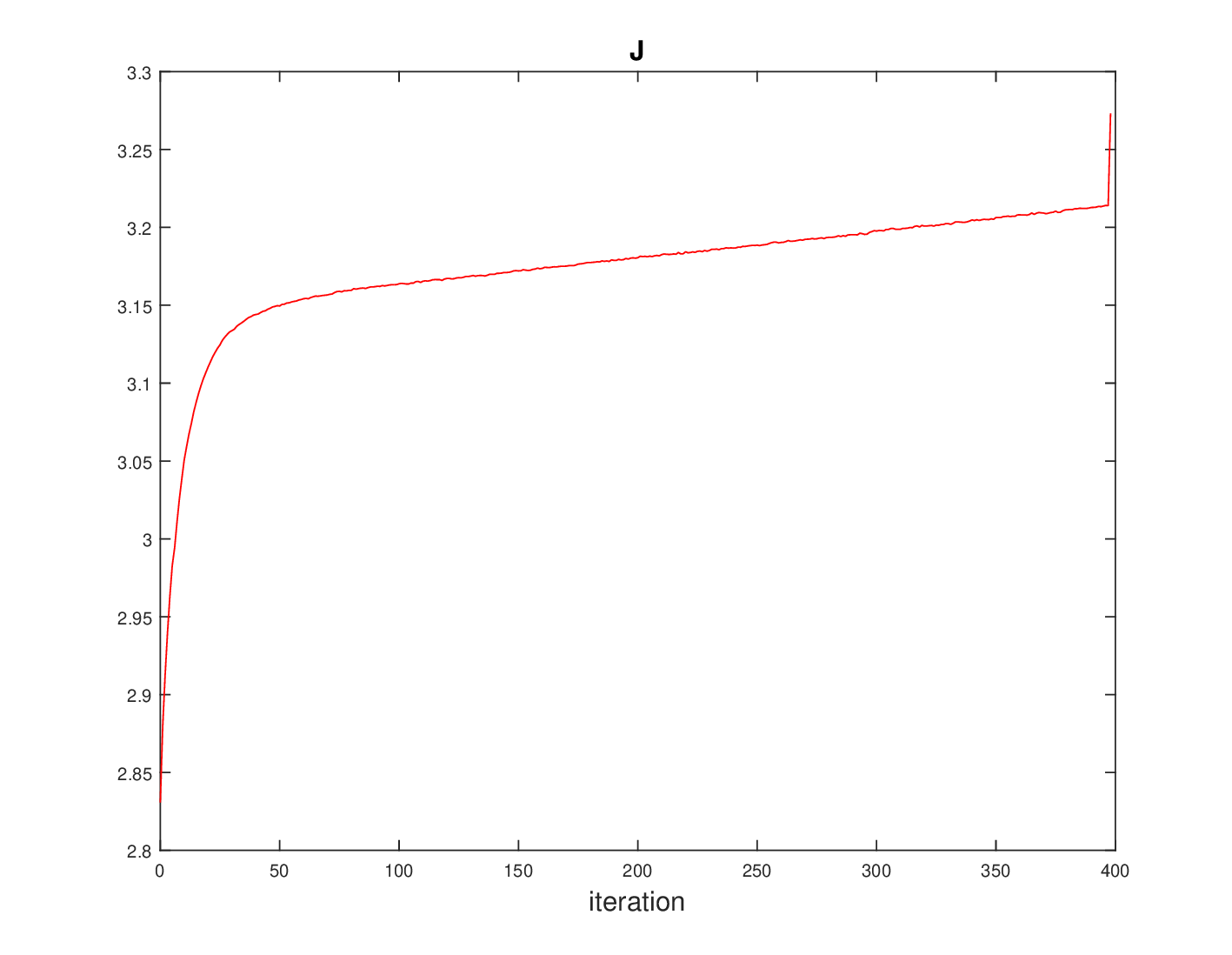}
    \caption{Value of the objective functional. Note the significant jump at the end, which occurs after refinement and 30 additional iterations of the homogenization process.}
    \label{fig:J}
\end{figure}
%In order to test previous algorithm we propose the following numerical test. Let $D=[-2.6,2.6]^2\setminus B((0,0);1)$ be a square without unit ball in center. We assume two multiple state equations $(m=2)$

\section*{Acknowledgements}
The research  has been supported in part by Croatian Science Foundation under the project IP-2022-10-7261. 
The calculations were performed at the Laboratory for Advanced Computing (Faculty of Science, University of Zagreb).

\printbibliography

\end{document}